\documentclass[reqno]{amsart}
\usepackage{amsmath,amsthm,amsfonts,amssymb,enumerate}
\usepackage[vcentermath,enableskew]{youngtab}
\usepackage{color}
\usepackage{microtype}
\usepackage{tikz}
\usepackage[colorlinks=true]{hyperref}
\usepackage[utf8]{inputenc}

\newtheorem{thm}{Theorem}[section]
\newtheorem*{thm*}{Theorem}
\newtheorem{lem}[thm]{Lemma}
\newtheorem*{lem*}{Lemma}
\newtheorem{cor}[thm]{Corollary}
\newtheorem{conj}{Conjecture}
\theoremstyle{definition} 
\theoremstyle{definition} \newtheorem{prop}[thm]{Proposition}
\theoremstyle{definition} \newtheorem{defn}[thm]{Definition}
\theoremstyle{definition} \newtheorem{rem}[thm]{Remark}
\theoremstyle{definition} 
\theoremstyle{definition} \newtheorem{ex}[thm]{Example}

\DeclareMathOperator{\sgn}{sgn}
\DeclareMathOperator{\GL}{GL}

\DeclareMathOperator{\Ind}{Ind}

\DeclareMathOperator{\shape}{shape}

\DeclareMathOperator{\Gr}{Gr}

\DeclareMathOperator{\Red}{Red}
\DeclareMathOperator{\red}{red}
\DeclareMathOperator{\fl}{fl}
\DeclareMathOperator{\hstrips}{hstrips}
\newcommand{\rowgrp}{\mathbf{R}}
\newcommand{\colgrp}{\mathbf{C}}

\newcommand{\Z}{\mathbb{Z}}
\newcommand{\N}{\mathbb{N}}
\newcommand{\C}{\mathbb{C}}
\renewcommand{\P}{\mathbb{P}}
\newcommand{\shapes}{M}  

\definecolor{light-gray}{gray}{.7}
\definecolor{dark-gray}{gray}{.3}

\begin{document}

\title[Patterns, Stanley symmetric functions and Specht modules]{Permutation Patterns, Stanley symmetric functions and generalized Specht modules}

\date{\today}
\author{Sara Billey and Brendan Pawlowski}
\thanks{Both authors were partially supported by grant DMS-1101017 from
the NSF} 

\begin{abstract}

Generalizing the notion of a vexillary permutation, we introduce a
filtration of $S_{\infty}$ by the number of terms in the Stanley
symmetric function, with the $k$th filtration level called the
$k$-vexillary permutations.  We show that for each $k$, the
$k$-vexillary permutations are characterized by avoiding a finite set
of patterns.  A key step is the construction of a Specht series, in
the sense of James and Peel, for the Specht module associated to the
diagram of a permutation. As a corollary, we prove a conjecture of Liu
on diagram varieties for certain classes of permutation diagrams.  We apply similar techniques to
characterize multiplicity-free Stanley symmetric functions, as well as
permutations whose diagram is equivalent to a forest in the sense of
Liu.
\end{abstract}

\maketitle

\section{Introduction}
\label{sec:intro}

In \cite{stanleysymm}, Stanley defined a symmetric function $F_w$
depending on a permutation $w$, with the property that the coefficient
of $x_1 \cdots x_{\ell}$ in $F_w$ is the number of reduced words of
$w$. Therefore, if $F_w = \sum_{\lambda} a_{w\lambda} s_{\lambda}$ is
written in terms of Schur functions, then
\begin{equation} \label{eq:reducedwordcount}
|\Red(w)| = \sum_{\lambda} a_{w\lambda} f^{\lambda},
\end{equation}
where $f^{\lambda}$ is the number of standard Young tableaux of shape
$\lambda$ and $\Red(w)$ the set of reduced words of $w$.

Edelman and Greene \cite{edelmangreene} gave an algorithm which
realizes (\ref{eq:reducedwordcount}) bijectively and shows that the
$a_{w\lambda}$ are nonnegative.  An alternative approach can be given
in terms of the nil-plactic monoid \cite{L-S-flag-variety-cohomology-hopf-structure}.

\begin{thm*} Given a permutation $w$, there is a set $\mathcal{EG}(w)$ of semistandard Young tableaux and a bijection
\begin{equation*}
\Red(w) \leftrightarrow \{(P,Q) : P \in \mathcal{EG}(w), \text{$Q$ a standard tableau of shape $\shape(P)$}\}.
\end{equation*}
\end{thm*}
The tableaux $\mathcal{EG}(w)$ are those semistandard tableaux whose column word---obtained by reading up columns starting with the leftmost---is a reduced word for $w$. The (transposed) shapes of these tableaux precisely give the Schur function expansion of $F_w$:
\begin{equation*}
F_w = \sum_{P \in \mathcal{EG}(w)} s_{\shape(P)^t},
\end{equation*}
where $\lambda^t$ is the conjugate of $\lambda$. Define the
permutation statistic $EG(w) = \sum a_{w\lambda}=|\mathcal{EG}(w)|$, which we call the \textit{Edelman-Greene number}.   

Stanley also characterized those $w$ for which $F_w$ is a single Schur function, or equivalently for which $EG(w)=1$: these are the \emph{vexillary} permutations, those avoiding the pattern $2143$. Our main results can be viewed as generalizations of this characterization. The first main theorem shows that $\mathcal{EG}(w)$ is well-behaved with respect to pattern containment.

\begin{thm} \label{thm:patternthm}
Let $v, w$ be permutations with $w$ containing $v$ as a pattern. There is an injection $\iota : \mathcal{EG}(v) \hookrightarrow \mathcal{EG}(w)$ such that if $P \in \mathcal{EG}(v)$, then $\shape(P) \subseteq \shape(\iota(P))$. Moreover, if $P, P'$ have the same shape, so do $\iota(P), \iota(P')$.
\end{thm}

Let $S_{\infty}=\bigcup_{n\geq 0} S_n$.  An immediate corollary is
that the sets $\{w \in S_{\infty } : EG(w) \leq k\}$ respect pattern
containment, in the sense that if $EG(w) \leq k$ and $w$ contains
$v$, then $EG(v) \leq k$. Our second main result is a sort of
converse.

\begin{defn} Given a positive integer $k$, a permutation $w \in S_n$ is \emph{$k$-vexillary} if $EG(w) \leq k$. \end{defn}

For example, the 1-vexillary permutations are the vexillary
permutations.  More information about these permutations and their
enumeration can be found in the Online Encyclopedia of Integer
Sequences (OEIS) entry A005802.  The number of $k$-vexillary
permutations in $S_{n}$ for $k=2,3,4$ appear in the OEIS as A224318,
A223034, A223905.  

Like the vexillary permutations, the $k$-vexillary permutations can be
characterized by permuation patterns.  This is our first main theorem. 

\begin{thm} \label{thm:kvexthm}
For each integer $k \geq 1$, there is a finite set $V_k$ of permutations such that $w$ is $k$-vexillary if and only if $w$ avoids all patterns in $V_k$.
\end{thm}

For the 2-vexillary and 3-vexillary permutations, we have explicitly
identified the list of patterns characterizing these sets.  We use
these properties to prove a conjecture of Ricky Liu on diagram
varieties related to 3-vexillary permutation diagrams.  We note that
permutation diagrams correspond with forests in the sense of Liu if
and only if the permutation avoids 4 patterns.  Furthermore, we can
give a nice description of Fulton's essential set for 3-vexillary
permutations.

Schur positive expansions of symmetric functions which are
multiplicity free have been important in many cases related to
representation theory and algebraic geometry.  For example, the Pieri
rule for multiplying a Schur function times a Schur function with just
one row or column is multiplicity free.  More generally, Stembridge
addressed the question of when the product of two Schur functions have
a multiplicity free expansion \cite{Stembridge.2001}.  Thomas and Yong
refined this work further in \cite{Thomas.Yong.2010}.  

As a corollary of Theorem~\ref{thm:patternthm}, we show that the
multiplicity free Stanley symmetric functions are indexed by a set of
permutations closed under taking patterns.  We conjecture that these
multiplicity free permutations can be characterized by avoiding a
finite set of permutations in $S_{6} \cup \cdots \cup S_{11}$.  As
with $k$-vexillary permutations, one can define a filtration on
permutations by bounding the multiplicities in the Stanley symmetric
functions.  It is shown that each filtration level again respects
pattern containment.  These permutations are also related to a new
type of pattern on the code of a permutation.  We also note that
3-vexillary permutations are multiplicity free.

In Section~\ref{sec:background}, we recall the connection between Stanley symmetric functions and the representation theory of the symmetric group, along with the Lascoux-Sch\"utzenberger recurrence for computing Stanley symmetric functions. We also recall the definitions of pattern avoidance and
containment. In Section~\ref{sec:JPmoves}, we introduce the notion of a James-Peel tree for a general diagram following \cite{jamespeel}, and prove a new decomposition theorem for general Specht modules based on Pieri's rule. Section~\ref{sec:transitions} specializes these ideas to permutation diagrams, with the Lascoux-Sch\"utzenberger tree as a key tool, and we prove Theorem~\ref{thm:patternthm}. In Section~\ref{sec:kvex}, we analyze in more detail the relationship between $EG(w)$ and $EG(v)$ for $v$ a pattern in $w$, and prove Theorem~\ref{thm:kvexthm}. Section~\ref{sec:diagvars} gives an application of Theorem~\ref{thm:patternthm} to computing the cohomology class of certain subvarieties of Grassmannians related to a conjecture of Ricky Liu.  In Section~\ref{sec:mult.free}, the multiplicity free and multiplicity bounded permutations are discussed. Section~\ref{sec:future} is devoted to open problems.

\bigskip

\section{Background}
\label{sec:background}

\subsection{Permutation patterns}

We first recall the definitions of pattern avoidance and containment for permutations.

\begin{defn} Let $x = x(1) \cdots x(n)$ be a sequence of distinct integers. The \emph{flatten map} $fl$ is defined by letting $fl(x)$ be the unique $v \in S_n$ such that $x(i) < x(j)$ if and only if $v(i) < v(j)$. \end{defn}

\begin{defn} A permutation $w$ \emph{contains} a permutation $v$ if there are $i_1 < \cdots < i_k$ such that $fl(w(i_1) \cdots w(i_k)) = v$. If $w$ does not contain $v$, then $w$ \emph{avoids} $v$. Frequently we call the smaller permutation $v$ a \emph{pattern} which $w$ contains or avoids. \end{defn}

\begin{ex} The permutation $2513764$ contains the patterns $2143$ (e.g. as the subsequence $2174$) and $23154$. It avoids $1234$. \end{ex}

\subsection{Specht modules}

Our proof of Theorem~\ref{thm:patternthm} goes via the representation
theory of $S_n$, specifically the interpretation of $F_w$ as the
Frobenius characteristic of a certain \emph{generalized Specht
module}, which we discuss next.  We assume the reader is familiar with
the classical $S_{n}$ representation theory described beautifully in
\cite{sagan}.

\begin{defn} A \emph{diagram} is a finite subset of $\N \times \N$. \end{defn}

We refer to the elements of a diagram as \emph{cells}. The diagrams of greatest interest for us will be \emph{permutation diagrams} (sometimes called Rothe diagrams, from \cite{rothe}). Define the diagram of a permutation $w \in S_n$ by
\begin{equation*}
D(w) = \{(i,w(j)) : 1 \leq i < j \leq n, w(i) > w(j)\}.
\end{equation*}
We'll draw $D(w)$ using matrix coordinates:

\begin{equation*}
D(243165) = \begin{array}{cccccc}
\circ & \times & \cdot & \cdot & \cdot & \cdot\\
\circ & \cdot & \circ & \times & \cdot & \cdot\\
\circ & \cdot & \times & \cdot & \cdot & \cdot\\
\times & \cdot & \cdot & \cdot & \cdot & \cdot\\
\cdot & \cdot & \cdot & \cdot & \circ & \times\\
\cdot & \cdot & \cdot & \cdot & \times & \cdot
\end{array}
\end{equation*}

Members of a diagram will be represented by $\circ$. We'll often
augment $D(w)$ by adding $\times$ at the points $(i, w(i))$.  By
definition, no member of the diagram lies directly below or directly right of an
$\times$.

A \emph{filling} of a diagram $D$ is a bijection $T : D \to \{1, \ldots, n\}$, where $n = |D|$. There is a natural left action of $S_n$ on fillings of $D$ by permuting entries. The \emph{row group} $\rowgrp(T)$ of a filling $T$ is the subgroup of $S_n$ consisting of permutations $\sigma$ which act on $T$ by permuting entries within their row; the \emph{column group} $\colgrp(T)$ is defined analogously. Define the \emph{Young symmetrizer} of a filling $T$ by
\begin{equation}\label{eq:young.sym}
y_T = \sum_{p \in \rowgrp(T)} \sum_{q \in \colgrp(T)} \sgn(q) qp,
\end{equation}
an element of $\C[S_n]$.

\begin{defn} Given a diagram $D$ and a choice of filling $T$, the
 \emph{Specht module} $S^D$ is the $S_n$-module $\C[S_n]y_T$, where $n
= |D|$. 
The  \emph{Schur function} $s_D$ of $D$ is the Frobenius
characteristic of $S^D$.
\end{defn}

\begin{rem} 
This definition generalizes the familiar definitions when $D$ is the
Ferrers diagram of a partition.  For general $D$, there is no known
expression $s_D = \sum_T x^T$ with $T$ running over some nice set of
fillings of $D$. When $D$ is a permutation diagram,
\cite{balancedlabellings} shows that the set of \emph{balanced
labellings} of $D$ works, but we will not need this fact. \end{rem}

Replacing $T$ with a different filling amounts to conjugating
$\rowgrp(T), \colgrp(T)$, and $y_T$, so the isomorphism type of $S^D$
is independent of the choice of $T$. Reordering the rows and columns
of $D$ also leads to an isomorphic Specht module, so we make the
following definition.

\begin{defn}
If a diagram $D$ is obtained from a diagram $D'$ by permuting rows and columns, say $D$ and $D'$ are \emph{equivalent}, and write $D \simeq D'$. This includes inserting or deleting empty rows and columns.
\end{defn}

A partition $\lambda = (\lambda_1 \geq \cdots \geq \lambda_\ell \geq 1)$ has an associated diagram
\begin{equation*}
\{(i,j) : 1 \leq i \leq \ell, 1 \leq j \leq \lambda_i\},
\end{equation*}
its \emph{Ferrers diagram}, which we will also denote by $\lambda$. Over $\C$, the Specht modules of Ferrers diagrams form complete sets of irreducible $S_n$-representations. For more on these classical irreducible Specht modules, see \cite{youngtableaux} or \cite{sagan}. In general, it is an open problem to find a reasonable combinatorial algorithm for decomposing $S^D$ into irreducibles. Reiner and Shimozono do so in \cite{percentavoiding} for \emph{percent-avoiding} diagrams $D$: those with the property that if $(i_1, j_1), (i_2, j_2) \in D$ with $i_1 > i_2$, $j_1 < j_2$, then at least one of $(i_1, j_2)$ and $(i_2, j_1)$ is in $D$. This includes the class of skew shapes and permutation diagrams. In a different direction, Liu \cite{liuforests} decomposes $S^D$ when $D$ is a diagram corresponding in a certain sense to a forest (see Section \ref{sec:diagvars}).

\subsection{Stanley symmetric functions}

Every permutation $w$ can be written as a product of adjacent
transpositions $s_{i}=(i, i+1)$.  Let $\ell(w)$ be the minimal length
of any such product.  Let $\Red(w)$ be the collection of
\textit{reduced words} for $w$.  Thus if $a=(a_{1},a_{2},\ldots ,
a_{\ell(w)}) \in \Red (w)$ then $s_{a_{1}} s_{a_{2}} \cdots
s_{a_{\ell(w)}}=w$ and this is a  minimal length expression for $w$.

Given a reduced word $a \in \Red(w)$, let $I(a)$ be the set of integer
sequences $1 \leq i_1 \leq \cdots \leq i_{\ell(w)}$ such that if $a_j
< a_{j+1}$, then $i_j < i_{j+1}$.

\begin{defn} The \emph{Stanley symmetric function} of $w$ is\footnote{Stanley's original function $G_w$ is our $F_{w^{-1}}$.}
\begin{equation*}
F_w = \sum_{a \in \Red(w)} \sum_{i \in I(a)} x_{i_1} \cdots x_{i_{\ell(w)}}.
\end{equation*}
\end{defn}
It is shown in \cite{stanleysymm} that $F_w$ is indeed symmetric.  For
a permutation $w$, let $1^m \times w = 12\cdots
m(w(1)+m)(w(2)+m)\cdots$. The results of \cite{billeyjockuschstanley}
show that $F_w = \lim_{m \to \infty} \mathfrak{S}_{1^m \times w}$,
where $\mathfrak{S}_v$ is a Schubert polynomial as defined by Lascoux
and Sch\"utzenberger in \cite{lascouxschutzenbergertree}.  The same
result can also be seen by decomposing a Schubert polynomial into key
polynomials using the nilplactic monoid
\cite{L-S-flag-variety-cohomology-hopf-structure}.  This implies
$F_{w} = F_{1^{m} \times w}$ for all $m\geq 1$.  Theorem 31 in
\cite{plactification} and Theorem 20 in \cite{percentavoiding} then
imply the following result, which is also implicit in
\cite{kraskiewicz}.

\begin{thm}
For any permutation $w$, $F_w = s_{D(w)}$.
\end{thm}

Stanley symmetric functions can be decomposed into Schur functions using a recursion introduced in \cite{lascouxschutzenbergerschubert,lascouxschutzenbergertree}. Given a permutation $w$, let $r$ be maximal with $w(r) > w(r+1)$. Then let $s > r$ be maximal with $w(s) < w(r)$. Let $t_{ij}$ denote the transposition $(i\,j)$, and define
\begin{equation*}
T(w) = \{wt_{rs}t_{rj} : \ell(wt_{rs}t_{rj}) = \ell(w) \text{ for some $j$}\};
\end{equation*}
or, if the set on the right-hand side is empty, set $T(w) = T(1 \times w)$. The members of $T(w)$ are called \emph{transitions} of $w$. The \emph{Lascoux-Sch\"utzenberger tree} (L-S tree for short) is the finite rooted tree of permutations with root $w$ where the children of a vertex $v$ are:
\begin{itemize}
\item None, if $v$ is vexillary (avoids $2143$).
\item $T(v)$ otherwise.
\end{itemize}
The finiteness of this tree is not immediately obvious
\cite{lascouxschutzenbergerschubert}, see
Remark~\ref{rem:finite-LS-tree} for a short proof. More on the
Lascoux-Sch\"utzenberger tree and its relationship to Schubert
polynomials and Stanley symmetric functions can be found in
\cite{manivel}.

\begin{ex} \label{ex:LStree}
The Lascoux-Sch\"utzenberger tree of $321465$ is
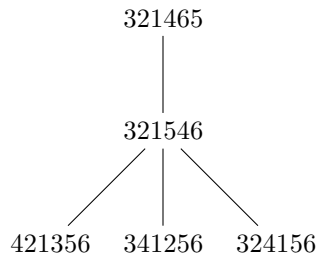
\begin{center}
\begin{tikzpicture}
\node { $321465$ }
 child {
  node {$321546$}
  child { node {$421356$} }
  child { node {$341256$} }
  child { node {$324156$} }
 }
;
\end{tikzpicture}
\end{center}
\end{ex}

Monk's rule for Schubert polynomials and the identity $F_w = \lim_{m \to \infty} \mathfrak{S}_{1^m \times w}$ lead to the recurrence
\begin{equation} \label{eq:transition-recurrence}
F_w = \sum_{v \in T(w)} F_v.
\end{equation}
This, together with the finiteness of the Lascoux-Sch\"utzenberger
tree terminating in vexillary leaves, and the fact that $F_v$ is a
Schur function exactly when $v$ is vexillary, imply that
\begin{equation*}
F_w = s_{D(w)} = \sum_v s_{\shape(v)},
\end{equation*}
where $v$ runs over the leaves of the L-S tree, and $\shape(v)$ denotes the partition whose shape is equivalent to $D(v)$. Here we use the fact that $D(v)$ is equivalent to a partition diagram if and only if $v$ is vexillary \cite{M2}. 

Note that upon taking coefficients of $x_1 x_2 \cdots x_{\ell}$ in the transition recurrence (\ref{eq:transition-recurrence}), one obtains $|\Red(w)| = \sum_{v \in T(w)} |\Red(v)|$. Little \cite{little-bijection} gives a bijective proof of this equality.

\begin{rem} The reduced words of $1\times w$ are exactly those of $w$
with all letters shifted up by $1$, and it is known that the same is
true of the tableaux in $\mathcal{EG}(1\times w)$ compared to the tableaux in
$\mathcal{EG}(w)$ since the algorithm only depends on the relative sizes of the
letters in the reduced words \cite{edelmangreene}. In particular, the
multiset of shapes are the same and $F_w = F_{1 \times w}$. Since the
L-S tree is finite, there is some $m$ such that
in constructing the tree for $1^m \times w$, we never need to make the replacement of $v$ by $1 \times v$. Thus we will ignore this
possible step in what follows.
\end{rem}

\bigskip

\section{James-Peel moves and subdiagrams}
\label{sec:JPmoves}

Let $D$ be a diagram. Given two positive integers $a, b$, let $R_{a\to b}D$ be the diagram which contains a cell $(i,j)$ if and only if one of the following cases holds:

\begin{itemize}
\item $i \neq a, b$ and $(i,j) \in D$.
\item $i = b$ and either $(a,j) \in D$ or $(b,j) \in D$.
\item $i = a$ and both $(a,j), (b,j) \in D$.
\end{itemize}

That is, $R_{a\to b}D$ is obtained by moving cells in row $a$ to row $b$ if the appropriate position is empty. Similarly, we define $C_{c \to d}D$ by moving cells of $D$ in column $c$ to column $d$ if possible. For example,
\begin{equation*}
D = \begin{array}{cccc}
\cdot & \circ & \cdot & \circ\\
\cdot & \cdot & \circ & \circ
\end{array} \qquad
R_{2\to 1}D = \begin{array}{cccc}
\cdot & \circ & \circ & \circ\\
\cdot & \cdot & \cdot & \circ
\end{array}
\end{equation*}

We also define $R_{a\to b}T$ and $C_{c\to d}T$ for a filling $T$, in the same way. From here through the proof of Theorem~\ref{thm:subdiagramstaircase}, we always view $S^D$, $S^{R_{a\to b}D}$, $S^{C_{c\to d}D}$ as the specific left ideals in $\C[S_{|D|}]$ generated by $y_T$, $y_{R_{a\to b}T}$, $y_{C_{c\to d}T}$ for a fixed filling $T$ of $D$ following the notation in Section~\ref{sec:background}.

We'll call the operators $R_{a\to b}$ and $C_{c\to d}$
\emph{James-Peel moves}, thanks to this theorem due to James and
Peel.

\begin{thm}\cite[Theorem 2.4]{jamespeel} \label{thm:JP} Let $(i_1, j_1), (i_2, j_2) \in D$ be such that $(i_1, j_2), (i_2, j_1) \notin D$. Let $D_R = R_{i_1 \to i_2}D$ and $D_C = C_{j_1 \to j_2}D$. Then there is a surjective homomorphism $\phi : S^D \twoheadrightarrow S^{D_R}$ with $S^{D_C} \subseteq \ker \phi$.
\end{thm}

We prove a generalization of this statement, and for the proof we will need more explicit knowledge of the homomorphism $\phi$. Given $(i_1, j_1)$, $(i_2, j_2)$ as in Theorem~\ref{thm:JP}, write $T_R = R_{i_1 \to i_2}T$ and $T_C = C_{j_1 \to j_2}T$. Let $Y$ and $Z$ be sets of coset representatives in  $\colgrp(T_C)$ and $\rowgrp(T_R)$ respectively such that
\begin{align*}
\colgrp(T_C) &= Y(\colgrp(T_C) \cap \colgrp(T))\\
\rowgrp(T_R) &= (\rowgrp(T_R) \cap \rowgrp(T))Z.
\end{align*}
Define $\phi$ to be right multiplication by $\sum_{\pi \in Z}\pi$. Then Theorem~\ref{thm:JP} follows from these identities using the Young
symmetrizers \eqref{eq:young.sym}:
\begin{enumerate}[(a)]
\item $\displaystyle y_T \sum_{\pi \in Z} \pi = y_{T_R}$ (implies $\phi(S^D) = S^{D_R}$)
\item $\displaystyle \sum_{\pi \in Y} \sgn(\pi)\pi \cdot y_T = y_{T_C}$ (implies $S^{D_C} \subseteq S^D$)
\item $\displaystyle y_{T_C} \sum_{\pi \in Z} \pi = 0$ (implies $S^{D_C} \subseteq \ker \phi$).
\end{enumerate}

\begin{rem} \label{rem:JPinclusion}
Only (c) above depends on the existence of a pair of cells $(i_1, j_1), (i_2, j_2)$ as in Theorem~\ref{thm:JP}. For arbitrary $a, b, c, d$ we still get a surjection $S^D \twoheadrightarrow S^{R_{a\to b}D}$ from (a), and a containment $S^{C_{c\to d}D} \subseteq S^D$ from (b). Over $\C$, we also get an inclusion $S^{R_{a\to b}D} \hookrightarrow S^D$.
\end{rem}

\begin{lem} \label{lem:commutingJPmoves} Suppose $R_{a\to b}C_{c\to d}D = C_{c\to d}R_{a\to b}D$. Let
\begin{align*}
&\phi : S^D \twoheadrightarrow S^{R_{a\to b}D}\\
&\phi' : S^{C_{c \to d}D} \twoheadrightarrow S^{R_{a\to b}C_{c \to d}D}
\end{align*}
be the surjections constructed above. Then
\begin{equation*}
\phi' = \phi|_{S^{C_{c\to d}D}}.
\end{equation*}
\end{lem}

\begin{proof}
Fix a filling $T$ of $D$ and take sets of coset representatives $Z, Z'$ with
\begin{align*}
&\rowgrp(R_{a \to b}T) = (\rowgrp(R_{a\to b}T) \cap \rowgrp(T))Z\\
&\rowgrp(R_{a\to b}C_{c\to d}T) = (\rowgrp(R_{a\to b}C_{c\to d}T) \cap \rowgrp(C_{c \to d}T))Z'
\end{align*}
so that $\phi, \phi'$ are right multiplication by $\sum_{\pi \in Z} \pi$ and $\sum_{\pi \in Z'} \pi$ respectively.

Applying a move $C_{c \to d}$ to a filling does not affect its row group, so
\begin{align*}
\rowgrp(R_{a\to b}T) = \rowgrp(C_{c \to d}R_{a\to b}T) &= \rowgrp(R_{a\to b}C_{c \to d}T)\\
 &= (\rowgrp(R_{a\to b}C_{c\to d}T) \cap \rowgrp(C_{c \to d}T))Z'\\
           &= (\rowgrp(C_{c\to d}R_{a\to b}T) \cap \rowgrp(C_{c \to d}T))Z'\\
           &= (\rowgrp(R_{a\to b}T) \cap \rowgrp(T))Z'.
\end{align*}
Thus we can take $Z' = Z$.
\end{proof}

\begin{defn}
A subset $D'$ of a diagram $D$ is a \emph{subdiagram} if it is the intersection of some rows and columns with $D$. That is, there are sets $U, V \subseteq \N$ such that $D' = (U \times V) \cap D$.
\end{defn}

Given two diagrams $D_1, D_2$ with $D_1 \subseteq [r] \times [c]$, let
\begin{equation*}
D_1 \cdot D_2 = D_1 \cup \{(i+r, j+c) : (i,j) \in D_2\}.
\end{equation*}

In this language, Theorem~\ref{thm:JP} applies when we have $(1)\cdot(1)$ as a subdiagram in $D$. Our generalization of Theorem~\ref{thm:JP} applies to a subdiagram of the form $ (p-1, p-2, \ldots, 1) \cdot (1)$. To simplify indexing, we will assume without loss of generality that our subdiagram occurs in rows $1, \ldots, p$ and columns $1, \ldots, p$. Write $\delta_p$ for the staircase shape $(p-1, p-2, \ldots, 1)$.

\begin{thm} \label{thm:subdiagramstaircase} Suppose $D$ contains $\delta_p \cdot (1)$ as a subdiagram in rows $1, \ldots, p$ and columns $1, \ldots, p$. There is a filtration
\begin{equation*}
0 = M_0 \subseteq M_1 \subseteq \cdots \subseteq M_p = S^D
\end{equation*}
of $S^D$ by $S_{|D|}$-submodules such that for each $1 \leq j \leq p$, there is a surjection
\begin{equation*}
M_j / M_{j-1} \twoheadrightarrow S^{R_{p\to p-j+1}C_{p\to j}D}.
\end{equation*}
\end{thm}

\begin{proof}
Let $F_j = C_{p \to j}D$ and $G_j = R_{p \to p-j+1}C_{p\to j}D$. Set
\begin{equation*}
M_j = \sum_{i=1}^j S^{F_i} \subseteq S^D,
\end{equation*}
with the containment by Theorem~\ref{thm:JP}.

Consider, for each $j$, the two surjections
\begin{align*}
&\phi_j : S^D \twoheadrightarrow S^{R_{p \to p-j+1}D}\\
&\theta_j : S^{F_j} \twoheadrightarrow S^{G_j}
\end{align*}
given by Theorem~\ref{thm:JP}. We have $R_{p\to p-j+1}C_{p\to j}D = C_{p\to j}R_{p\to p-j+1}D$. Indeed, this commutation property depends only on the subdiagram of $D$ in rows $p-j+1, p$ and columns $j, p$. By hypothesis this subdiagram is
\begin{equation*}
\begin{array}{cc}
\cdot & \cdot\\
\cdot & \circ
\end{array}
\end{equation*}
and either order of James-Peel moves results in the subdiagram
\begin{equation*}
\begin{array}{cc}
\circ & \cdot\\
\cdot & \cdot
\end{array}\ .
\end{equation*}
Therefore, Lemma~\ref{lem:commutingJPmoves} says that $\theta_j = \phi_j|_{S^{F_j}}$.

If $1 \leq i < j$, then $(i, p-j+1), (p, p) \in D$ and $(i, p), (p, p-j+1) \notin D$, so Theorem~\ref{thm:JP} implies that $S^{F_i} \subseteq \ker \phi_j$, hence $M_{j-1} \subseteq \ker \phi_j$. Thus, $S^{F_j} \cap M_{j-1} \subseteq S^{F_j} \cap \ker \phi_j = \ker \theta_j$, so $\theta_j$ descends to a surjection
\begin{equation*}
S^{F_j} / (S^{F_j} \cap M_{j-1}) \twoheadrightarrow S^{G_j}.
\end{equation*}
Since there is a canonical isomorphism
\begin{equation*}
M_j/M_{j-1} \simeq S^{F_j} / (S^{F_j} \cap M_{j-1})
\end{equation*}
given by $m + M_{j-1} \mapsto m + S^{F_j} \cap M_{j-1}$ where $m \in S^{F_j}$, we are done.

\end{proof}

\begin{rem}\label{rem:specht.series}
Theorem~\ref{thm:JP} and hence Theorem~\ref{thm:subdiagramstaircase}
are actually valid over any field, and lead to the existence of
\emph{Specht series} for certain Specht modules.  A \emph{Specht
series} for an $S_n$-module $M$ is a filtration $0 = M_0 \subseteq M_1
\subseteq \cdots \subseteq M_N = M$ where each quotient $M_{i+1}/M_i$
is isomorphic to a (classical) Specht module $S^{\lambda}$. Over $\C$
these are just composition series, but in general they are coarser,
since Specht modules are indecomposable but not necessarily
irreducible in finite characteristic.
\end{rem}

We won't need this level of generality, so from now on we will work
over $\C$ and freely split exact sequences. In particular,

\begin{cor} \label{cor:subdiagramstaircase}
If $D$ contains $\delta_p \cdot (1)$ as a subdiagram in rows $1, \ldots, p$ and columns $1, \ldots, p$, then we have the inclusion 
\begin{equation}\label{eq:subdiag}
\bigoplus_{j=1}^p S^{R_{p\to p-j+1}C_{p\to j}D} \hookrightarrow S^D
\end{equation}
as $S_{|D|}$-modules over $\mathbb{C}$.
\end{cor}
Observe that $C_{j\to j}D = R_{j\to j}D=D$ for all $j$ and $D$, so for $j=1$
and $j=p$ above only one move changes the diagram.   

\begin{ex}
Take
\begin{equation*}
D = D(4261735) = \begin{array}{ccccc}
\circ & \circ & \circ & \cdot & \cdot\\
\circ & \cdot & \cdot & \cdot & \cdot\\
\circ & \cdot & \circ & \cdot & \circ\\
\cdot & \cdot & \cdot & \cdot & \cdot\\
\cdot & \cdot & \circ & \cdot & \circ\\
\end{array}
\end{equation*}
where we have omitted the last empty rows and columns. The subdiagram
in rows $1, 2, 5$ and columns $1, 2, 5$ is $(2,1) \cdot (1)$.  The
following diagrams appear in \eqref{eq:subdiag}:
\begin{equation*}
R_{5\to 1}D = \begin{array}{ccccc}
\circ & \circ & \circ & \cdot & \circ\\
\circ & \cdot & \cdot & \cdot & \cdot\\
\circ & \cdot & \circ & \cdot & \circ\\
\cdot & \cdot & \cdot & \cdot & \cdot\\
\cdot & \cdot & \circ & \cdot & \cdot\\
\end{array} \quad
R_{5\to 2}C_{5\to 2} D = \begin{array}{ccc}
\circ & \circ & \circ\\
\circ & \circ & \circ\\
\circ & \circ & \circ\\
\end{array} \quad
C_{5\to 1} D = \begin{array}{ccccc}
\circ & \circ & \circ & \cdot & \cdot\\
\circ & \cdot & \cdot & \cdot & \cdot\\
\circ & \cdot & \circ & \cdot & \circ\\
\cdot & \cdot & \cdot & \cdot & \cdot\\
\circ & \cdot & \circ & \cdot & \cdot\\
\end{array}
\end{equation*}
Corollary~\ref{cor:subdiagramstaircase} now says $S^{(3,3,3)} \oplus S^{R_{5\to 1}D} \oplus S^{C_{5 \to 1}D} \hookrightarrow S^{D}$. Applying Theorem~\ref{thm:JP} to the cells $(2,1)$, $(5,3)$ in $R_{5\to 1}D$ gives $S^{(4,3,2)} \oplus S^{(4,3,1,1)} \hookrightarrow S^{R_{5\to 1}D}$. Using the cells $(1,2), (3,5)$ in $C_{5 \to 1}D$, Theorem~\ref{thm:JP} gives $S^{(4,2,2,1)} \oplus S^{(3,3,2,1)} \hookrightarrow S^{C_{5 \to 1}D}$. In fact, all these inclusions are isomorphisms 
\begin{equation*}
S^{D} \simeq S^{(3,3,3)} \oplus S^{(4,3,2)} \oplus S^{(4,3,1,1)} \oplus S^{(4,2,2,1)} \oplus S^{(3,3,2,1)},
\end{equation*}
as one can check using the Lascoux-Sch\"utzenberger tree, or by computing the Edelman-Greene tableaux of $4261735$:
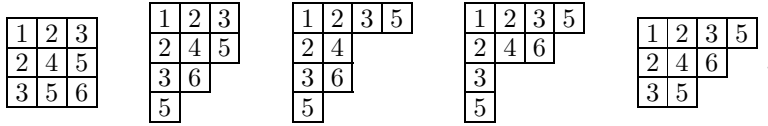
\begin{equation*}
\young(123,245,356) \qquad \young(123,245,36,5) \qquad \young(1235,24,36,5) \qquad \young(1235,246,3,5) \qquad \young(1235,246,35)\,\,.
\end{equation*}

This example also provides a case where Theorem~\ref{thm:subdiagramstaircase} is more powerful than Theorem~\ref{thm:JP}. For reasons which will become clear in Section~\ref{sec:transitions}, we would like to apply James-Peel moves which apply to the cell $(5,5)$, as the first stage of moves above do. There are $3$ possible cells one could pair $(5,5)$ with in Theorem~\ref{thm:JP}: $(1,1)$, $(1,2)$, and $(2,1)$. However, one finds that in each case the inclusion $S^{C_{c\to d}D} \oplus S^{R_{a\to b}D} \hookrightarrow S^D$ is not an isomorphism, so $S^D$ cannot be completely decomposed using such moves.
\end{ex}

Note that when Corollary~\ref{cor:subdiagramstaircase} is applied to the diagram $\delta_p \cdot (1)$ itself, the resulting partitions are exactly those arising from applying Pieri's rule to expand $s_{\delta_p} s_{(1)}$ in terms of Schur functions. Indeed, it follows readily from the group algebra definitions that $S^{D_1 \cdot D_2} \simeq \Ind_{S_{|D_1|} \times S_{|D_2|}}^{S_{|D_1| + |D_2|}} S^{D_1} \otimes S^{D_2}$, and hence that $s_{D_1 \cdot D_2} = s_{D_1} s_{D_2}$. We can therefore view Corollary~\ref{cor:subdiagramstaircase} as applying Pieri's rule to a subdiagram of $D$. See the discussion surrounding Example~\ref{ex:subdiagrampieri} for an expansion on this idea.

\begin{defn}
Given a diagram $D$ contained in $[m] \times [n]$, define
\begin{align*}
D^{\max} &= (R_{m\to 1} \cdots R_{2 \to 1})(R_{m \to 2} \cdots R_{3 \to 2}) \cdots (R_{m \to m-1})D\\
D^{\min} &= (C_{n\to 1} \cdots C_{2 \to 1})(C_{n \to 2} \cdots C_{3 \to 2}) \cdots (C_{n \to n-1})D.
\end{align*}
\end{defn}

These diagrams are both equivalent to partitions---an identification we will freely make---and satisfy $S^{D^{\max}}, S^{D^{\min}} \hookrightarrow S^{D}$ by Remark~\ref{rem:JPinclusion}. To be precise, $D^{\min}$ is equivalent to the partition gotten by sorting the row lengths of $D$ (the number of cells in each row), and $D^{\max}$ to the conjugate of the partition gotten by sorting the column lengths. This second description implies the following lemma.

\begin{lem} \label{lem:canonicaldiagrams} If $\mathbf{I}$ is any
sequence of James-Peel row moves $R_{a \to b}$, then
$D^{\max}=(\mathbf{I}D)^{\max}$.  Likewise, if $\mathbf{J}$ is a
sequence of column moves $C_{c \to d}$, then $D^{\min}=
(\mathbf{J}D)^{\min}$.
\end{lem}

The partitions $D^{\min}$ and $D^{\max}$ play a special role in the structure of $S^D$. In the case $D = D(w)$, the next lemma is Theorem~4.1 from \cite{stanleysymm}.

\begin{lem} \label{lem:minmaxpartitions} Let $D$ be any diagram and $\lambda$ a partition. If $S^{\lambda} \hookrightarrow S^D$, then $D^{\min} \leq \lambda \leq D^{\max}$ in dominance order. Also, $S^{D^{\min}}$ and $S^{D^{\max}}$ appear in $S^D$ with multiplicity one. \end{lem}

\begin{proof} We will prove the part of the statement referring to $D^{\min}$, with the proof for $D^{\max}$ being analogous. Induct on $\ell(D^{\min})$, the number of non-empty rows of $D$. The lemma is obvious when $D$ is a single row. Let $H$ be a row of $D$ with minimal length, and set $E = D \setminus H$. Then $D$ is the image of $E \cdot H$ under James-Peel moves---push $H$ to its original row, then each cell individually to its original column---so $S^D \hookrightarrow S^{E \cdot H}$.

Suppose $S^{\mu} \hookrightarrow S^D \hookrightarrow S^{E \cdot
H}$. Since $s_{E \cdot H} = s_E s_H$, Pieri's rule says $\mu$ is
obtained from some $\lambda$ with $S^{\lambda} \hookrightarrow S^E$
by adding $|H|$ cells to it, no two in the same column. By
construction, $D^{\min}$ is obtained by appending a row of length $|H|$
to $E^{\min}$. The induction hypothesis gives $E^{\min} \leq
\lambda$. To show $D^{\min} \leq \mu$, let $\alpha,\beta$ be the
partitions equivalent to $D^{\min}$ and $E^{\min}$.  Observe that for
all $1\leq m \leq \ell(E^{\min})$,
\begin{equation}\label{eq:dom}
\sum_{i=1}^m \alpha_i = \sum_{i=1}^m \beta_i \leq \sum_{i=1}^m
\lambda_i \leq \sum_{i=1}^m \mu_i,
\end{equation}
while if $m = \ell(E^{\min})+1 = \ell(D^{\min})$, both sides of \eqref{eq:dom} are equal
to $|D|$, since $\ell(\mu) \leq \ell(\lambda)+1 \leq
\ell(E^{\min})+1$.

By induction, $S^{E^{\min}}$ appears with multiplicity one in
$S^E$. Since Pieri's rule is multiplicity free, $S^{D^{\min}}$ appears
with multiplicity one in $S^{E \cdot H}$, hence with multiplicity at
most one in $S^D$. Furthermore, $S^{D^{\min}}$ does actually appear in
$S^D$, because $D^{\min}$ is the image of $D$ under James-Peel moves,
and so it appears with multiplicity one.
\end{proof}

James-Peel moves and Corollary~\ref{cor:subdiagramstaircase} present one possible way to decompose a Specht module into irreducibles. In general it is not known if an arbitrary Specht module can be decomposed by finding some appropriate tree of James-Peel moves, as the inclusion in Corollary~\ref{cor:subdiagramstaircase} may not be an isomorphism. 
The way we prove Theorem~\ref{thm:patternthm} is to find such a decomposition for the case of $D(w)$. The usefulness of James-Peel moves for us comes from the fact that they are well-behaved with respect to subdiagram inclusion, and pattern inclusion for permutations corresponds to subdiagram inclusion on the level of permutation diagrams.

To be more precise about this, we make the following definition.

\begin{defn}  \label{defn:JPtree} A \emph{James-Peel tree} for a diagram $D$ is a rooted tree $\mathcal{T}$ with vertices labeled by diagrams and edges labeled by sequences of James-Peel moves, satisfying the following conditions:
\begin{itemize}
\item The root of $\mathcal{T}$ is $D$.
\item If $B$ is a child of $A$ with a sequence $\mathbf{JP}$ of James-Peel moves labeling the edge $A\text{---}B$, then $B = \mathbf{JP}(A)$.
\item If $A$ has more than one child, these children arise as a result of applying Corollary~\ref{cor:subdiagramstaircase} to $A$. That is, $A$ contains $\delta_p \cdot (1)$ as a subdiagram in rows $i_1 < \cdots < i_p$ and columns $j_1 < \cdots < j_p$, and each edge leading down from $A$ is labeled $R_{i_p\to i_{p-k+1}}C_{j_p\to j_k}$ for some distinct values $1 \leq k \leq p$ (perhaps not all such $k$ appear).
\end{itemize}
\end{defn}

Note that the vertex labels are completely determined by the root and the edge labels. When a vertex is labeled by a permutation diagram $D(w)$, sometimes we will refer to it simply as $w$. See Examples~\ref{ex:subdiagrampieri} and \ref{ex:JPtree} later in this section for examples of James-Peel trees.

Corollary~\ref{cor:subdiagramstaircase} and Remark~\ref{rem:JPinclusion} immediately imply the following lemma.  

\begin{lem}\label{lem:JPtree}
If $D$ has a James-Peel tree $\mathcal{T}$ with leaves $A_1, \ldots, A_m$, then $\bigoplus_i S^{A_i} \hookrightarrow S^D$ 
as $S_{|D|}$-modules over $\mathbb{C}$.
\end{lem}

\begin{defn}
A James-Peel tree $\mathcal{T}$ for $D$ is \emph{complete} if its leaves $A_1, \ldots, A_m$ are equivalent to Ferrers diagrams of partitions and $S^D \simeq \bigoplus_i S^{A_i}$.
\end{defn}

In \cite{jamespeel}, an algorithm is given which constructs a complete James-Peel tree when $D$ is a skew shape. More generally, Reiner and Shimozono \cite{columnconvex} construct a complete James-Peel tree for any \emph{column-convex} diagram: a diagram $D$ for which $(a,x), (b,x) \in D$ with $a < b$ implies $(i,x) \in D$ for all $a < i < b$. In the next section we construct a complete James-Peel tree for the diagram of a permutation, so it's worth noting that neither of these classes of diagrams contains the other. For example, $D(37154826)$ is not equivalent to any column-convex or row-convex diagram, while the column-convex diagram
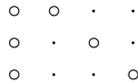
\begin{equation*}
\begin{array}{cccc}
\circ & \circ & \cdot & \cdot\\
\circ & \cdot & \circ & \cdot\\
\circ & \cdot & \cdot & \circ
\end{array}
\end{equation*}
is not equivalent to the diagram of any permutation. The James-Peel
trees constructed in \cite{jamespeel} and \cite{columnconvex} are
binary trees based on moves from Theorem~\ref{thm:JP}.  By
Corollary~\ref{cor:subdiagramstaircase}, the James-Peel trees
constructed here do not need to be binary, a vertex can have an
arbitrary number of children.

\begin{rem} \label{rem:spechtseries}
Theorem~\ref{thm:subdiagramstaircase} shows that a complete
James-Peel tree for $D$ yields a Specht series for $S^D$ over any
field.  In particular, Theorem~\ref{thm:permutationJPtree} below shows
that $S^{D(w)}$ always has a Specht series.
\end{rem}

\begin{defn} \label{inducedJPtree}
Given a James-Peel tree $\mathcal{T}'$ for a subdiagram $D'$ of $D$, the \emph{induced} James-Peel tree $\mathcal{T}$ for $D$ is defined as follows. Start with $\mathcal{T}$ an unlabeled tree isomorphic to $\mathcal{T}'$, with $\phi : \mathcal{T}' \to \mathcal{T}$ an isomorphism. Give each edge $\phi(A_1) \text{---} \phi(A_2)$ of $\mathcal{T}$ the same label as the edge $A_1 \text{---} A_2$ of $\mathcal{T}'$. Label the root $\phi(D')$ of $\mathcal{T}$ with $D$, and label the rest of the vertices according to the James-Peel moves labeling the edges in $\mathcal{T}$. 
\end{defn} 

Observe that the first two conditions of Definition~\ref{defn:JPtree}
clearly hold for $\mathcal{T}$ as constructed.  The subdiagram of $D'$
needed in the third condition for $\mathcal{T}'$ and $D'$ works just
as well for $\mathcal{T}$ and $D$, when viewed as a subdiagram of $D$.
Thus, $\mathcal{T}$ is a James-Peel tree for $D$.

The notion of an induced James-Peel tree provides a convenient way to discuss a generalization of Theorem~\ref{thm:subdiagramstaircase} from the case of a subdiagram $\delta_p \cdot (1)$ to that of any subdiagram $\lambda \cdot (k)$ with $\lambda$ a partition. Recall the classical \emph{Pieri rule}:
\begin{equation*}
s_{\lambda} s_{(k)} = \sum_{\mu} s_{\mu},
\end{equation*}
where $\mu$ runs over all partitions gotten by adding $k$ cells to $\lambda$, no two in the same column. That is, let $\hstrips_k(\lambda)$ be the set of length $\ell(\lambda)+1$ compositions $\alpha$ of $k$ such that $\alpha_i \leq \lambda_i - \lambda_{i-1}$ for $i > 1$. Then $s_{\lambda} s_{(k)} = \sum_{\alpha \in \hstrips_k(\lambda)} s_{\lambda + \alpha}$, where $\lambda + \alpha$ is entrywise addition.

The moves in Theorem~\ref{thm:subdiagramstaircase} can be thought of as realizing Pieri's rule on $\delta_p \cdot (1)$ in terms of James-Peel moves. Suppose we have a James-Peel tree $\mathcal{T}$ for $\lambda \cdot (k)$ whose leaves are the partitions $\lambda + \alpha$ for $\alpha \in \hstrips_k(\lambda)$. If $D$ contains $\lambda \cdot (k)$ as a subdiagram, we can take the James-Peel tree for $D$ induced by $\mathcal{T}$, which amounts to realizing Pieri's rule on the subdiagram $\lambda \cdot (k)$ using James-Peel moves, generalizing Theorem~\ref{thm:subdiagramstaircase}.

In fact we only need the case $\lambda = \delta_p$, $k = 1$, so rather than state and prove a precise theorem, we will be content with giving an example of such a tree.

\begin{ex} \label{ex:subdiagrampieri}
Take $\lambda = (3,1,1)$, $k = 2$. In each non-leaf vertex, we have shaded the cells to which Theorem~\ref{thm:subdiagramstaircase} is being applied.

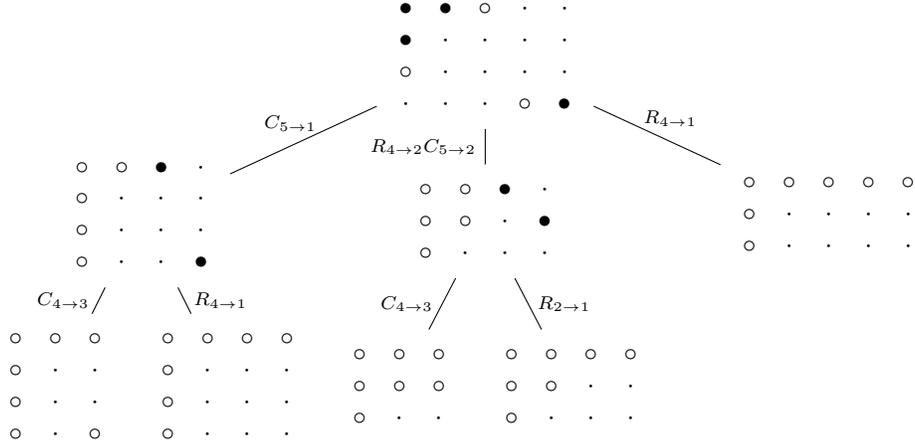
\begin{center}
\begin{tikzpicture}[level 1/.style={sibling distance=13em}, level 2/.style={sibling distance=6.5em}]
\node { $\begin{array}{ccccc}
\bullet & \bullet & \circ & \cdot & \cdot\\
\bullet & \cdot & \cdot & \cdot & \cdot\\
\circ & \cdot & \cdot & \cdot & \cdot\\
\cdot & \cdot & \cdot & \circ & \bullet
\end{array}$ }
  child[level distance=6em] {
   node { $\begin{array}{cccc}
\circ & \circ & \bullet & \cdot\\
\circ & \cdot & \cdot & \cdot\\
\circ & \cdot & \cdot & \cdot\\
\circ & \cdot & \cdot & \bullet
\end{array}$ } 
   child[level distance=6.5em] {
    node { $\begin{array}{ccc}
\circ & \circ & \circ\\
\circ & \cdot & \cdot\\
\circ & \cdot & \cdot\\
\circ & \cdot & \circ
\end{array}$ } edge from parent node[left] {$\scriptstyle C_{4\to 3}$}
   }
   child[level distance=6.5em] {
    node { $\begin{array}{cccc}
\circ & \circ & \circ & \circ\\
\circ & \cdot & \cdot & \cdot\\
\circ & \cdot & \cdot & \cdot\\
\circ & \cdot & \cdot & \cdot
\end{array}$ } edge from parent node[right] {$\scriptstyle R_{4\to 1}$}
   } edge from parent node[above] {$\scriptstyle C_{5\to 1}\quad$}
  }  
  child[level distance=6.25em] {
   node { $\begin{array}{cccc}
\circ & \circ & \bullet & \cdot\\
\circ & \circ & \cdot & \bullet\\
\circ & \cdot & \cdot & \cdot
\end{array}$ }
   child {
    node { $\begin{array}{ccc}
\circ & \circ & \circ\\
\circ & \circ & \circ\\
\circ & \cdot & \cdot
\end{array}$ } edge from parent node[left] {$\scriptstyle C_{4\to 3}$}
   }
   child {
    node { $\begin{array}{cccc}
\circ & \circ & \circ & \circ\\
\circ & \circ & \cdot & \cdot\\
\circ & \cdot & \cdot & \cdot
\end{array}$ } edge from parent node[right] {$\scriptstyle R_{2\to 1}$}
   } edge from parent node[left] {$\scriptstyle R_{4\to 2}C_{5\to 2}$}
  }
  child[level distance=6em] {
   node { $\begin{array}{ccccc}
\circ & \circ & \circ & \circ & \circ\\
\circ & \cdot & \cdot & \cdot & \cdot\\
\circ & \cdot & \cdot & \cdot & \cdot
\end{array}$  } edge from parent node[above] {$\quad \scriptstyle R_{4\to 1}$}
  } 
;
\end{tikzpicture}
\end{center}
\end{ex}

The main example of induced James-Peel trees for us will come from permutation patterns. The connection is that if $w$ contains a pattern $v$, then $D(v)$ is (up to reindexing) a subdiagram of $D(w)$. Specifically, if the pattern $v$ appears in positions $i_1, \ldots, i_k$ of $w$, then $D(v)$ is the subdiagram of $D(w)$ induced by the rows $i_1, \ldots, i_k$ and columns $w(i_1), \ldots, w(i_k)$.

Let $\shapes(D)$ denote the multiset of partitions of $n = |D|$ such
that $S^D \simeq \bigoplus_{\lambda \in \shapes(D)}
S^{\lambda}$. Given partitions $\lambda, \mu$, let $\lambda + \mu$ be
the partition $(\lambda_1 + \mu_1, \lambda_2 + \mu_2, \ldots)$,
padding $\lambda$ or $\mu$ with $0$'s as necessary. Let $\lambda \cup
\mu$ be the partition whose parts are the (multiset) union of the
parts of $\lambda$ and those of $\mu$. 

\begin{lem} \label{lem:JPtreetwosubdiagrams}
Suppose $D_1, D_2$ are subdiagrams of $D$, each with a complete James-Peel tree, and such that $D_1 = (U \times V) \cap D$ and $D_2 = (U^c \times V^c) \cap D$. Let $F_1 = (U^c \times V) \cap D$ and $F_2 = (U \times V^c) \cap D$. Then there is a well-defined injection $\iota : \shapes(D_1) \times \shapes(D_2) \to \shapes(D)$ given by
\begin{equation*}
\iota(\lambda, \mu) = (\lambda \cup F_1^{\min}) + (F_2^{\max} \cup \mu).
\end{equation*}
\end{lem}

\begin{proof}
Without loss of generality we can assume $U=\{1,2,\ldots,i\}$ and
$V=\{1,2,\ldots, j \}$ by permuting rows and columns of $D$ if
necessary.  Let $\mathcal{T}_1, \mathcal{T}_2$ be complete James-Peel
trees for $D_1, D_2$.  Then, we can further assume that the leaves of
$\mathcal{T}_1, \mathcal{T}_2$ are all partition diagrams by doing
extra James-Peel moves to sort the rows and columns.

Let $\mathcal{T}$  be the James-Peel tree for $D$ induced from 
$\mathcal{T}_1$, with $\phi : \mathcal{T}_1 \to \mathcal{T}$ an
isomorphism as in Definition~\ref{inducedJPtree}. Since $D$ contains
$D_1 = \phi^{-1}(D)$ as a subdiagram, each vertex $A$ of $\mathcal{T}$
contains $\phi^{-1}(A)$ as a subdiagram. In particular, each leaf $B$
of $\mathcal{T}$ has the block form
\begin{equation}\label{eq:block.form}
B=\begin{tabular}{|c|c|}
\hline
$\lambda$ & $F_2'$\\
\hline
$F_1'$     & $D_2$\\
\hline
\end{tabular}\ ,
\end{equation}
where  $\lambda$ is
the shape of $\phi^{-1}(B)$,\ $F_1'$ is the image of $F_1$ under moves $C_{c\to d}$, and
$F_2'$ the image of $F_2$ under moves $R_{a \to b}$.  

Using the block form \eqref{eq:block.form}, we next add a single child
to each leaf $B$ of $\mathcal{T}$.  By Lemma
\ref{lem:canonicaldiagrams}, $(F_1')^{\min} = (F_1)^{\min}$ and
$(F_2')^{\max} = (F_2)^{\max}$. Thus, there is a sequence
$\mathbf{I}_{B}$ of upward row moves involving only rows in $U$, and a
sequence $\mathbf{J}_{B}$ of leftward column moves involving only
columns in $V$, such that $\mathbf{J}_{B}(F_1') = F_1^{\min}$ and
$\mathbf{I}_{B}(F_2') = F_2^{\max}$.  Since the upper-left block is
a partition diagram, it is unaffected by the James-Peel moves
$\mathbf{I}_{B}$ and $\mathbf{J}_{B}$.  Since no cell of $D_2$ lies in
a row in $U$ or a column in $V$, $\mathbf{J}_{B}$ and $\mathbf{I}_{B}$
do not change $D_2$ either. Thus, we can  define 
\begin{equation*} \widetilde{B} = \mathbf{I}_{B}\mathbf{J}_{B} (B) =
\begin{tabular}{|c|c|}
\hline
$\lambda$ & $F_2^{\max}$\\
\hline
$F_1^{\min}$     & $D_2$\\
\hline
\end{tabular}\ .
\end{equation*}
To each leaf $B$ of $\mathcal{T}$, attach the child $\widetilde{B}$
via an edge labeled $\mathbf{I}_{B}\mathbf{J}_{B}$.  Note, the result
is still a James-Peel tree for $D$.  We will abuse notation and again
call this tree $\mathcal{T}$.

We modify $\mathcal{T}$ one more time by augmenting each leaf with an
induced tree for $D_{2}$.  Specifically, to each leaf $\widetilde{B}$
of $\mathcal{T}$, attach the James-Peel tree for $\widetilde{B}$
induced by $\mathcal{T}_2$.   As above, each leaf $C$ of the new tree descending
from $\widetilde{B}$ now has block form
\begin{equation}\label{eq:block.form.2}
C=\begin{tabular}{|c|c|}
\hline
$\lambda$ & $F_{2}''$\\
\hline
$F_{1}''$     & $\mu$\\
\hline
\end{tabular}\ ,
\end{equation}
where $\lambda,\mu$ are a pair of shapes in $\shapes(D_1) \times
\shapes(D_2)$,\ $F_{1}''$ is the result of applying row moves to
$F_1^{\min}$ and $F_{2}''$ is the result of applying column moves to
$F_2^{\max}$. 
 Notice that the upper-right and lower-left block of
$\widetilde{B}$ and $C$ are equivalent, since both are equivalent
to partitions.  The upper-left block of $\widetilde{B}$ and $C$ are
exactly the same since the induced tree for $\widetilde{B}$ does not touch the
first $i$ rows and $j$ columns.  Again, we abuse notation by calling this
tree $\mathcal{T}$.

Finally, we modify $\mathcal{T}$ once again so the leaves all have
block form with 4 partition shapes.  Assume $C$ is a descendant of $\widetilde{B}$
with block diagonal shapes $\lambda,\mu$ as in \eqref{eq:block.form.2}
above.  Let $\mathbf{I}_{C}$ be the sequence of upward James-Peel row
moves needed to sort the rows of $F_1''$ into the partition shape
$F_1^{\min}$, and let $\mathbf{J}_{C}$ be the sequence of leftward
column moves needed to sort the columns of $F_2''$ into the partition
shape $F_2^{\max}$.  Note, such moves will not change the shapes
$\lambda$ and $\mu$ when applied to $C$ since they are partitions.
Thus, one can define
\begin{equation}\label{eq:block.form.3}
\widetilde{C}:=\mathbf{I}_{C}\mathbf{J}_{C}(C) = \begin{tabular}{|c|c|}
\hline
$\lambda$ & $F_2^{\max}$\\
\hline
$F_1^{\min}$     & $\mu$\\
\hline
\end{tabular}\ ,
\end{equation}
with all four subdiagrams equal to honest left- and top-justified Ferrers diagrams.
For each leaf $C$ of $\mathcal{T}$, attach
$\widetilde{C}=\mathbf{I}_{C}\mathbf{J}_{C}(C)$ as a child.  

Observe that the resulting tree $\mathcal{T}$ is a James-Peel tree for
$D$, and the leaves are in bijection with the multiset $\shapes
(D_{1}) \times \shapes (D_{2})$.  One can also see that if a leaf
$\widetilde{C}$ of $\mathcal{T}$ has diagonal shapes $\lambda ,\mu $
then
\begin{equation*}
\widetilde{C}^{\max} = (\lambda \cup F_1^{\min}) + (F_2^{\max} \cup \mu)
\end{equation*}
and the shape of $\widetilde{C}^{\max}$ only depends on $\lambda ,\mu,
F_1, F_2$ and not on $B$ or $C$.  Thus, we define $ \iota(\lambda,
\mu)$ to be the partition of shape $\widetilde{C}^{\max}$ which is in
$\shapes(D)$ by Lemma~\ref{lem:minmaxpartitions} and
Lemma~\ref{lem:JPtree}.  This gives a well-defined injection of
multisets $\iota : \shapes(D_1) \times \shapes(D_2) \to \shapes(D)$ as
intended.
\end{proof}

For the most part we will only need a simpler version of this lemma.

\begin{cor} \label{cor:JPtreesubdiagram} Suppose $D$ has a subdiagram $D'$ with a complete James-Peel tree. There is an injection $\iota : \shapes(D') \hookrightarrow \shapes(D)$ such that $\lambda \subseteq \iota(\lambda)$. Moreover, $\iota(\lambda)$ depends only on $\lambda$: if $\lambda$ appears $k$ times in $\shapes(D')$, then $\iota(\lambda)$ appears at least $k$ times in $\shapes(D)$. \end{cor}

In particular, taking $D = D(w)$ and $D' = D(v)$ for $v$ a pattern in $w$, Corollary~\ref{cor:JPtreesubdiagram} together with the equalities
\begin{equation*}
s_{D(w)} = F_{w} = \sum_{P \in \mathcal{EG}(w)} s_{\shape(P)^{t}}
\end{equation*}
will immediately imply Theorem~\ref{thm:patternthm} once we show that $D(w)$ always has a complete James-Peel tree.

\begin{ex} \label{ex:JPtree}
Take the diagram $D = D(316524)$.  Rows $5,6$ and column $6$ are
empty, so we won't draw them in the following picture
\begin{equation*}
D = \begin{array}{cccc|c}
\circ & \circ & \cdot & \cdot & \cdot\\
\cdot & \cdot & \cdot & \cdot & \cdot\\
\cdot & \circ & \cdot & \circ & \circ\\
\hline
\cdot & \circ & \cdot & \circ & \cdot
\end{array}.
\end{equation*}
Let $D_{1}$ be the subdiagram on rows 1,2,3 and columns 1,2,3,4 which
corresponds to the pattern $31524 = \fl(31624)$, so $D_1 \simeq
D(31524)$.

A complete James-Peel tree $\mathcal{T}_{1}$ for $D_1$ is
\begin{center}
\begin{tikzpicture}
\node {$D_1$}
 child { node {$A_1$}
         edge from parent
         node[left] {$\scriptstyle R_{3\to 1}$}
        }
 child { node {$A_2$}
         edge from parent
         node[right] {$\scriptstyle C_{4\to 1}$}
        }
;
\end{tikzpicture}
\end{center}
where $A_1 \simeq (3,1), A_2 \simeq (2,2)$.
The James-Peel tree $\mathcal{T}_{1}$ for $D_1$ induces the following James-Peel tree $\mathcal{T}$ 
for $D$ 

\begin{center}
\begin{tikzpicture}
\node {$D$}
 child { node {$B_{31}$}
         edge from parent
         node[left] {$\scriptstyle R_{3\to 1}$}
        }
 child { node {$B_{22}$}
         edge from parent
         node[right] {$\scriptstyle C_{4\to 1}$}
        }
;
\end{tikzpicture}
\end{center}
where $B_{31} = R_{3\to 1}D$ and $B_{22} = C_{4\to 1}D$ with 

\begin{equation*}
B_{31} = \begin{array}{cccc|c}
\circ & \circ & \cdot & \circ & \circ\\
\cdot & \cdot & \cdot & \cdot & \cdot\\
\cdot & \circ & \cdot & \cdot & \cdot\\
\hline
\cdot & \circ & \cdot & \circ & \cdot
\end{array} \qquad
B_{22} = \begin{array}{cccc|c}
\circ & \circ & \cdot & \cdot & \cdot\\
\cdot & \cdot & \cdot & \cdot & \cdot\\
\circ & \circ & \cdot & \cdot & \circ\\
\hline
\circ & \circ & \cdot & \cdot & \cdot
\end{array}
\end{equation*}

Following the proof of Lemma~\ref{lem:JPtreetwosubdiagrams}, we next apply leftward column moves to the lower left subdiagrams $F_1'$ to get $(F_1')^{\min}$, and upward row moves to the upper right subdiagrams $(F_2')$ to get $(F_2')^{\max}$:

\begin{equation*}
\widetilde{B}_{31} = C_{4\to 1}B_{31} = \begin{array}{cccc|c}
\circ & \circ & \cdot & \circ & \circ\\
\cdot & \cdot & \cdot & \cdot & \cdot\\
\cdot & \circ & \cdot & \cdot & \cdot\\
\hline
\circ & \circ & \cdot & \cdot & \cdot
\end{array} \qquad
\widetilde{B}_{22} = R_{3\to 1}B_{22} = \begin{array}{cccc|c}
\circ & \circ & \cdot & \cdot & \circ\\
\cdot & \cdot & \cdot & \cdot & \cdot\\
\circ & \circ & \cdot & \cdot & \cdot\\
\hline
\circ & \circ & \cdot & \cdot & \cdot
\end{array}\ .
\end{equation*}

At this point we would apply James-Peel moves to the lower right subdiagram, but it is empty so there's nothing to do. Finally, we apply leftward column moves and rightward row moves to make all four subdiagrams into Ferrers diagrams (up to trailing empty rows and columns):
\begin{equation*}
\widetilde{C}_{31,\emptyset} = R_{3\to 2}C_{4\to 3}C_{2\to 1}\widetilde{B}_{31} = \begin{array}{cccc|c}
\circ & \circ & \circ & \cdot & \circ\\
\circ & \cdot & \cdot & \cdot & \cdot\\
\cdot & \cdot & \cdot & \cdot & \cdot\\
\hline
\circ & \circ & \cdot & \cdot & \cdot
\end{array} \qquad
\widetilde{C}_{22,\emptyset} = R_{3\to 2}\widetilde{B}_{22} = \begin{array}{cccc|c}
\circ & \circ & \cdot & \cdot & \circ\\
\circ & \circ & \cdot & \cdot & \cdot\\
\cdot & \cdot & \cdot & \cdot & \cdot\\
\hline
\circ & \circ & \cdot & \cdot & \cdot
\end{array}\ .
\end{equation*}
Now, taking unions before additions as the order of operations
\begin{align*}
&(\widetilde{C}_{31,\emptyset})^{\max} = (3,1) \cup (2) + (1) \cup \emptyset = (3,2,1) + (1) = (4,2,1),\\
&(\widetilde{C}_{22,\emptyset})^{\max} = (2,2) \cup (2) + (1) \cup \emptyset = (2,2,2) + (1) = (3,2,2).
\end{align*}
For the injection $\iota : \shapes(D_1) \hookrightarrow \shapes(D)$ we therefore take $\iota(3,1) = (4,2,1)$ and $\iota(2,2) = (3,2,2)$. Indeed, $s_{D_1} = F_{31524} = s_{32/1} = s_{31} + s_{22}$, while $s_D = F_{316524} = s_{322} + s_{331} + s_{421}$.

\end{ex}

\bigskip

\section{Transitions as James-Peel moves}
\label{sec:transitions}

Recall the following notation from Section~\ref{sec:background}. Given a permutation $w$, take $r$ maximal with $w(r) > w(r+1)$, then $s > r$ maximal with $w(s) < w(r)$. The set of transitions of $w$ is
\begin{equation} \label{eq:transitions}
T(w) = \{wt_{rs}t_{rj} : \ell(wt_{rs}t_{rj}) = \ell(w)\},
\end{equation}
or else $T(1\times w)$ if the set on the right is empty. Note that $wt_{rs}t_{rj} \in T(w)$ if and only if $w(j) < w(s)$ and there is no $j < j' < r$ with $w(j) < w(j') < w(s)$.

Upon taking diagrams of permutations, each transition corresponds to a sequence of James-Peel moves.

\begin{lem} \label{lem:JPtransitions} Given a permutation $w$, let $r, s$ be as above and take $w' = wt_{rs}t_{rj} \in T(w)$. Then
\begin{equation*}
D(w') = R_{r \to j} C_{w(s) \to w(j)} D(w) = C_{w(s) \to w(j)} R_{r \to j} D(w).
\end{equation*}
\end{lem}

\begin{proof}
We will show that the change in passing from $D(w)$ to $D(w')$ is as follows:
\begin{equation*}
\begin{array}{ccccc}
\,& w(j)    &                                                           & w(s)                                      & w(r)\\
j & \times &                     \cdots                                & \cdot                                    & \cdot\\
  & \vdots &                                                           & \colorbox{dark-gray}{\phantom{$\vdots$}} &     \\
r & \cdot  & \!\!\!\!\colorbox{light-gray}{\phantom{$\cdots$}}\!\!\!\! & \circ               & \times\\
  &        &                                                           &                     &\\
s & \cdot  &                                                           & \times              & \cdot
\end{array} \quad \xrightarrow{\hspace*{.5cm}} \quad
\begin{array}{ccccc}
\,& w(j)                                      &                                                           & w(s)    & w(r)\\
j & \circ                                    & \!\!\!\!\colorbox{light-gray}{\phantom{$\cdots$}}\!\!\!\! & \times & \cdot\\
  & \colorbox{dark-gray}{\phantom{$\vdots$}} &                                                           & \vdots  &      \\
r & \times                                   &           \cdots                                          & \cdot  & \cdot\\
  &                                           &                                                           &        &\\
s & \cdot                                    &                                                           & \cdot  & \times
\end{array}
\end{equation*}
where we move the cells in each shaded region of $D(w)$ into the corresponding (formerly cell-free) shaded region of $D(w')$, and also move the cell $(r, w(s))$, denoted by $\circ$ above, to $(j, w(j))$. Here only the region $[j,r] \times [w(j), w(s)]$ together with row $s$ and column $w(r)$ have been drawn.  We will show the rest of the diagram remains unchanged. We use $\cdots$ and $\vdots$ to denote a sequence of empty cells of arbitrary length.

Consider row-by-row the effect on diagrams of passing from $w$ to $w'$. It is clear that rows $k < j$ of $D(w')$ match those of $D(w)$. Rows $k > s$ also match: indeed, they are all empty.

In row $j$, by passing from $D(w)$ to $D(w')$ we could only gain cells. Specifically, a cell is gained in column $w(k)$ if and only if the following equivalent conditions hold:
\begin{itemize}
\item $w(j) < w(k) < w(s)$ and $k > j$, or $w(k) = w(j)$
\item $w(j) < w(k) < w(s)$ and $k > r$, or $w(k) = w(j)$
\item $(r, w(k)) \in D(w)$ and $w(j) < w(k)$, or $w(k) = w(j)$.
\end{itemize}

On the other hand, in row $r$, we could only lose cells. A cell is lost in column $w(k)$ if and only if following equivalent conditions hold:
\begin{itemize}
\item $w(j) < w(k) < w(r)$ and $k > r$
\item $w(j) < w(k) < w(s)$ and $k > r$, or $w(k) = w(s)$
\item $(r, w(k)) \in D(w)$ and $w(j) < w(k)$, or $w(k) = w(s)$.
\end{itemize}
Thus, the effect of passing from $w$ to $w'$ on rows $j$ and $r$ is to move all cells in row $r$ between columns $w(j)$ and $w(s)$ up to row $j$, and to move $(r,w(s))$ to $(j, w(j))$.

Now say $j < k < r$. The only column in which a cell could be gained in row $k$ is column $w(j)$, which happens if and only if the following equivalent conditions hold:
\begin{itemize}
\item $w(k) > w(j)$
\item $w(k) > w(s)$
\item $(k, w(s)) \in D(w)$.
\end{itemize}
Conversely, if there is a cell in row $k$ and column $w(s)$ of $D(w)$, it is lost in $D(w')$.

So at least within the region $[j,r] \times [w(j),w(s)]$, one does obtain $D(w')$ from $D(w)$ by performing the indicated James-Peel moves. To show that in fact $D(w') = R_{r \to j} C_{w(s) \to w(j)} D(w)$, we must show that these James-Peel moves do not move any cells outside of $[j,r] \times [w(j),w(s)]$. That is:
\begin{enumerate}[(i)]
\item If $(k, w(s)) \in D(w)$ for $k < j$ or $k > r$, then $(k, w(j)) \in D(w)$.
\item If $(r, w(k)) \in D(w)$ for $w(k) < w(j)$ or $w(k) > w(s)$, then $(j, w(k)) \in D(w)$.
\end{enumerate}
For (i), rows $k > r$ are empty, so assume $k < j$ and $(k, w(s)) \in D(w)$. Then $w(k) > w(s) > w(j)$ and $k < j$ give $(k, w(j)) \in D(w)$. For (ii), $(r, w(s))$ is the rightmost cell in row $r$ by the choice of $r, s$, so assume $w(k) < w(j)$ and $(r, w(k)) \in D(w)$. Then $(j, w(k)) \in D(w)$, because $j < r < k$ and $w(k) < w(j)$.
\end{proof}

\begin{thm} \label{thm:permutationJPtree}
For a permutation $w$, the diagram $D(w)$ has a complete James-Peel tree.
\end{thm}

\begin{proof}
If $w$ is vexillary, the tree with one vertex $D(w)$ and no edges is a
complete James-Peel tree for $D(w)$. Otherwise, let $v^1, \ldots, v^p$
be the transitions of $w$, say $v^i = wt_{rs}t_{rj_i}$ where $s > r > j_1 >
\cdots > j_p$. Then $w(j_1) < \cdots < w(j_p) < w(s) < w(r)$,
so $\fl(w(j_p)\cdots w(j_1)w(r)w(s)) = p\cdots 1(p+2)(p+1)$, and  $D(p\cdots
1(p+2)(p+1))$ is exactly $(p-1, \ldots, 1) \cdot (1)$ after removing an empty row and column. Thus, $D(w)$
contains $(p-1, \ldots, 1) \cdot 1$ as a subdiagram in rows $j_p,
\ldots, j_{2}, r$ and columns $w(j_1), \ldots, w(j_{p-1}), w(s)$.

Let
\begin{equation*}
D_i = \begin{cases}
R_{r \to j_i}C_{w(s) \to w(j_i)}D(w) & \text{if $1 < i < p$}\\
C_{w(s) \to w(j_1)}D(w) & \text{if $i = 1$}\\
R_{r \to j_p}D(w) & \text{if $i = p$}
\end{cases}
\end{equation*}
The diagrams $D_i$ are exactly those produced by Corollary~\ref{cor:subdiagramstaircase}, and $\bigoplus_{i=1}^p S^{D_i} \hookrightarrow S^{D(w)}$. 

Let $\mathcal{T}$ be the James-Peel tree with root $D(w)$ and children
$D_i$, where $D(w)$ is connected to $D_i$ by an edge labeled with the
appropriate James-Peel move(s).  Next, connect $D_1$ to a child $E_1 =
R_{r \to j_1}D_1$ by an edge labeled $R_{r \to j_1}$, and $D_p$ to a
child $E_p = C_{w(s) \to w(j_p)}D_p$ by an edge labeled $C_{w(s) \to
w(j_p)}$. The leaves $E_1, D_2, \ldots, D_{p-1}, E_p$ of $\mathcal{T}$
are now exactly $D(v^1), \ldots, D(v^p)$ by
Lemma~\ref{lem:JPtransitions}. Attach to each $D(v^i)$ the tree
inductively produced for $v^i$ by transitions and so on until every
leaf is the diagram of a vexillary permutation.

The tree $\mathcal{T}$ is still a James-Peel tree for $D(w)$. By
construction, its leaves are the diagrams of the leaves of the
Lascoux-Sch\"utzenberger tree of $w$. The equation $s_{D(w)} = \sum_v
s_{\shape(v)}$ from Section~\ref{sec:background}, with $v$ running
over the leaves of the L-S tree, implies that $\mathcal{T}$ is
complete.
\end{proof}

Let $JP(w)$ be the James-Peel tree for $D(w)$ constructed in Theorem~\ref{thm:permutationJPtree}.   
Corollary~\ref{cor:JPtreesubdiagram} and Theorem~\ref{thm:permutationJPtree} now yield our first main result.

\newtheorem*{thm:patternthmagain}{Theorem \ref{thm:patternthm}}

\begin{thm:patternthmagain}
Let $v, w$ be permutations with $w$ containing $v$ as a pattern. There is an injection $\iota : \mathcal{EG}(v) \hookrightarrow \mathcal{EG}(w)$ such that if $P \in \mathcal{EG}(v)$, then $\shape(P) \subseteq \shape(\iota(P))$. Moreover, if $P, P'$ have the same shape, so do $\iota(P), \iota(P')$.
\end{thm:patternthmagain}

\begin{cor} \label{cor:kvexrespectspatterns} If a permutation $w$ is $k$-vexillary and $v$ is a pattern in $w$, then $v$ is $k$-vexillary. \end{cor}

\begin{cor} \label{cor:kvexrespectsmultiplicity} If $v$ is a pattern in $w$ and $F_w$ is multiplicity-free, so is $F_v$. More generally, if $\langle F_w, s_{\lambda} \rangle \leq k$ for all $\lambda$ then $\langle F_v, s_{\mu} \rangle \leq k$ for all $\mu$. \end{cor}

\begin{rem}
Theorem~\ref{thm:patternthm} shows the existence of an injection $\mathcal{EG}(v) \hookrightarrow \mathcal{EG}(w)$ which respects inclusion of shapes for $v$ a pattern contained in $w$, but an explicit map on tableaux is lacking. The Edelman-Greene correspondence shows that this is equivalent to an injection $\Red(v) \hookrightarrow \Red(w)$ which is an inclusion on the shapes of Edelman-Greene insertion tableaux. Tenner's \cite{tennervexillary} characterization of vexillary permutations yields an explicit injection in the case where $v$ is vexillary. 
\end{rem}

\begin{rem}
We note that Crites, Panova and Warrington have studied the connection
between the shape of a permutation under the RSK correspondence and
pattern containment \cite{crites-panova-warrington}.  The injection
given in Theorem~\ref{thm:patternthm} on shapes is quite
different since the Edelman-Greene tableaux of a permutation are based
on the reduced words instead of the one-line notation.  At this time,
we don't know of a connection between their work and our injection.   
\end{rem}

So far we have only used Corollary~\ref{cor:JPtreesubdiagram}, but the full strength of Lemma~\ref{lem:JPtreetwosubdiagrams} yields another interesting result.

\begin{thm} Let $w \in S_n$ be a permutation and $I \subseteq [n]$. If $u_1$ is the subsequence of $w$ in positions $I$, and $u_2$ the subsequence in positions $[n] \setminus I$, then
\begin{equation*}
EG(w) \geq  EG(\fl(u_1)) \cdot EG(\fl(u_2)).
\end{equation*}
 \end{thm}

\bigskip

\section{$k$-vexillary permutations}
\label{sec:kvex}

In this section we show that the property of $k$-vexillarity is characterized by avoiding a finite set of patterns for any $k$. The key step is to remove some inessential moves from the James-Peel tree for $D(w)$, namely those which only permute rows or columns.

If $D$ is an arbitrary diagram, and $\sigma, \tau$ are permutations, let $(\sigma,\tau)D$ be the diagram $\{(\sigma(i), \tau(j)) : (i,j) \in D\}$. Given a James-Peel tree $\mathcal{T}$ for $D$, let $(\sigma,\tau)\mathcal{T}$ denote the James-Peel tree for $(\sigma,\tau)D$ gotten by replacing every James-Peel move $R_{x\to y}$ labeling an edge of $\mathcal{T}$ by $R_{\sigma(x) \to \sigma(y)}$, and every move $C_{x \to y}$ by $C_{\tau(x) \to \tau(y)}$, and relabeling vertices accordingly. Whenever a move labeling an edge $e$ of a James-Peel tree just permutes rows or columns, we can eliminate that move from the tree at the cost of relabeling rows and columns of James-Peel moves below $e$, as follows.

\begin{defn}
Given a James-Peel tree $\mathcal{T}$ of a diagram $D$, define the \emph{reduced James-Peel tree} $\red(\mathcal{T})$ of $D$ inductively.

\begin{itemize}
\item If $D$ has no children in $\mathcal{T}$, then $\red(\mathcal{T}) = \mathcal{T}$.
\item If $D$ has just one child $F$, and $D = (\sigma, \tau)F$ for some $\sigma, \tau \in S_{\infty}$, let $\mathcal{T}_1$ be the subtree of $\mathcal{T}$ below $F$ with root $F$. Then $\red(\mathcal{T}) = (\sigma, \tau)\red(\mathcal{T}_1)$.
\item If $D$ has at least two children $F_1, F_{2},\ldots, F_p$ or $D$ has one child
$F_{1}$ not equivalent to $D$, let
$\mathcal{T}_i$ be the subtree of $\mathcal{T}$ below $F_i$ with root
$F_i$. Then $\red(\mathcal{T})$ is $\mathcal{T}$ with each
$\mathcal{T}_i$ replaced by $\red(\mathcal{T}_i)$.
\end{itemize}
\end{defn}

\begin{defn} A rooted tree is \emph{bushy} if every non-leaf vertex has at least two children. \end{defn}

\begin{lem} If $\mathcal{T}$ is a complete James-Peel tree for $D$,
then $\red(\mathcal{T})$ is a complete James-Peel tree for $D$.
Furthermore, $\red(\mathcal{T})$ is bushy. 
\end{lem}

\begin{proof}
Note that $\red(\mathcal{T})$ is still a James-Peel tree for $D$. As equivalent diagrams have isomorphic Specht modules, if $\mathcal{T}$ is complete then so is $\red(\mathcal{T})$.

Next, for any vertex $A$ of $\mathcal{T}$, the subtree of
$\mathcal{T}$ below $A$ is itself a complete James-Peel tree for
$A$. In particular, $S^A$ is determined by the leaves below
it. Therefore, if $A$ has only a single child $B$ in $\mathcal{T}$,
then $S^A$ and $S^B$ are isomorphic.

Now suppose $\mathcal{T}$ is not bushy. The only way this can happen
is if $\mathcal{T}$ has a vertex $A$ with only one child $B$, but $A$
and $B$ are not equivalent. There is a James-Peel move relating $A, B$
(or a sequence of them, but we can consider them one at a time), say
$B = R_{a\to b}A$. If one of rows $a$ and $b$ of $A$ is contained in
the other, then $R_{a\to b}A$ is simply $A$ with those two rows
interchanged, so rows $a$ and $b$ are not comparable under inclusion since $A$ and $B$ are not equivalent.
There are cells $(a, j_1), (b, j_2) \in A$ with $(a, j_2), (b, j_1)
\notin A$. By Theorem~\ref{thm:JP}, $S^B \oplus S^{C_{j_1 \to j_2}A}
\hookrightarrow S^A$. As $S^{C_{j_1 \to j_2}A} \neq 0$, $S^B$ is not
isomorphic to $S^A$. This contradicts the previous paragraph, so
$\mathcal{T}$ must be bushy.
\end{proof}

\begin{lem} \label{lem:maxtreesize} The number of edges in a bushy tree with $k$ leaves is at most $2k-2$. \end{lem}

\begin{proof}
This follows by induction on the number of leaves.
\end{proof}

Recall $JP(w)$ is the James-Peel tree for $D(w)$ constructed in Theorem~\ref{thm:permutationJPtree}, and
let $RJP(w) = \red(JP(w))$. In the vicinity of a vertex $D(v)$, $JP(w)$ looks like this:
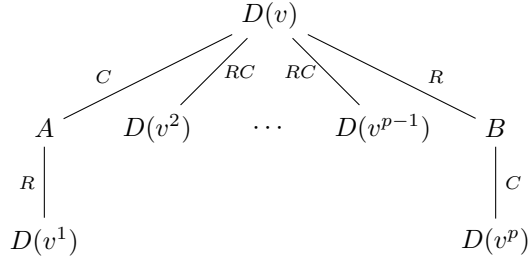
\begin{center}
\begin{tikzpicture}
\node { $D(v)$ }
 child { node { $A$ }
     child { node { $D(v^1)$ } edge from parent[left] node {$\scriptstyle R$} }
   edge from parent[left] node {$\scriptstyle C\quad$}
 }
 child { node { $D(v^2)$ } edge from parent[right] node {$\scriptstyle RC$} }
 child { node { $\cdots$ } edge from parent[style=white] }
 child { node { $D(v^{p-1})$ } edge from parent[left] node {$\scriptstyle  RC$} }
 child { node { $B$ }
     child { node { $D(v^p)$ } edge from parent[right] node {$\scriptstyle C$} }
   edge from parent[right] node {$\scriptstyle\quad R$}
 }
;
\end{tikzpicture}
\end{center}
Here the $v^i = vt_{rs}t_{rj_i}$ are the transitions of $v$, with $j_1
> \cdots > j_p$. The rows of $D(v)$ involved in row moves are $r, j_1,
\ldots, j_p$, and the columns involved in column moves are $v(s),
v(j_1), \cdots, v(j_p)$.  In the figure above, we have elongated two
of the paths for the proofs to come.

\begin{lem}\label{lem:equivalent} Suppose $v$ has transitions $v^1, \ldots, v^p$ as above. Then $D(v^1)$ is equivalent to $C_{v(s) \to v(j_1)}D(v)$, and $D(v^p)$ is equivalent to $R_{r \to j_p}D(v)$. \end{lem}

\begin{proof}
By Lemma~\ref{lem:JPtransitions}, 
\[
D(v^1) =  C_{v(s) \to v(j_1)} R_{r\to j_1} D(v) 
\]
and 
\[
D(v^p) = R_{r \to j_p} C_{v(s) \to v(j_p)} D(v).
\]
 It suffices to check that column $v(s)$ of $D(v)$ contains column $v(j_p)$, and that row $r$ contains row $j_1$. Suppose the first of these fails, that there is $(i, v(j_p)) \in D(v)$ with $(i, v(s)) \notin D(v)$. Choose the maximal such $i$. Then $vt_{rs}t_{ri}$ is a transition of $v$, which is impossible since $i < j_p$. The argument for the row containment is analogous.
\end{proof}

Thus, upon passing to $RJP(w)$, the edges $A\text{---}D(v^1)$ and $B\text{---}D(v^p)$ are contracted. For a diagram $D$, write $[D]$ for the equivalence class of diagrams containing $D$. We use this notation below when we have a diagram equivalent to $D$ but don't need to specify exactly what the diagram is. In the vicinity of a vertex $[D(v)]$, $RJP(w)$ has the form
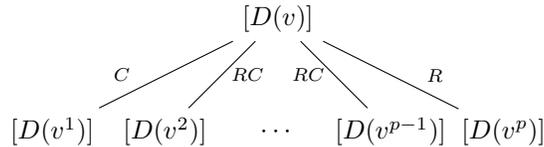
\begin{center}
\begin{tikzpicture}
\node { $[D(v)]$ }
 child { node { $[D(v^1)]$ } edge from parent[left] node {$\scriptstyle C\quad$} }
 child { node { $[D(v^2)]$ } edge from parent[right] node {$\scriptstyle RC$} }
 child { node { $\cdots$ } edge from parent[style=white] }
 child { node { $[D(v^{p-1})]$ } edge from parent[left] node {$\scriptstyle  RC$} }
 child { node { $[D(v^p)]$ } edge from parent[right] node {$\scriptstyle \quad R$} }
;
\end{tikzpicture}
\end{center}
Strictly speaking, we haven't shown that both moves on the edges labeled $RC$ survive in $RJP(w)$, but this won't be important. We therefore will speak of $R$-edges, $C$-edges, and $RC$-edges of $RJP(w)$, each non-leaf vertex having exactly one $R$-edge and one $C$-edge leading to children. 

Now suppose $\mathcal{T}$ is a subtree of $RJP(w)$ with the same
root. Let $R(\mathcal{T})$ be the union of $\{a,b\}$ over all $R_{a\to
b}$ appearing in $\mathcal{T}$, and $C(\mathcal{T})$ the union of
$\{c,d\}$ over all $C_{c\to d}$ appearing in $\mathcal{T}$. Write
$R(\mathcal{T}) \cup w^{-1}C(\mathcal{T}) = \{i_1 < \cdots < i_r\}$,
and define the permutation associated to this tree   
\begin{equation*}
w_{\mathcal{T}} = fl(w(i_1)\cdots w(i_r)).
\end{equation*}

\begin{rem} In Section~\ref{sec:background} we noted that, for convenience, $w$ could be replaced by $1^m \times w$ to remove the necessity of sometimes replacing $v$ by $1 \times v$ in the Lascoux-Sch\"{u}tzenberger tree. The definition of $w_{\mathcal{T}}$ above is then an abuse of notation, since we are really taking a subsequence of $1^m \times w$. However, rows and columns $1, \ldots, m$ of $D(w)$ are empty, so are not affected at all by the James-Peel moves in $RJP(w)$ or $\mathcal{T}$. This means that the subsequence defining $w_{\mathcal{T}}$ occurs entirely after the $m$th position of $1^m \times w$, so we are free to shift it down by $m$ and consider it as a subsequence of $w$. This applies also to Theorems \ref{thm:skeleton} and \ref{thm:kvex} below.
\end{rem}

We would like to bound the number of letters of $w_{\mathcal{T}}$ in terms of the number of leaves of $\mathcal{T}$. Such a bound depends on the sizes of $R(\mathcal{T})$ and $C(\mathcal{T})$, so the following definition is convenient to get good bounds.

\begin{defn} A subtree $\mathcal{T}$ of $RJP(w)$ with root $D(w)$ is \emph{colorful} if each
non-leaf vertex of $\mathcal{T}$ has at least the two children
corresponding to its $R$-edge and its $C$-edge.  Thus, colorful implies bushy.  
\end{defn}

\begin{lem} \label{lem:subtreepattern} Say $\mathcal{T}$ is a subtree of $RJP(w)$ rooted at $D(w)$ with $k$ leaves. Then $k \leq EG(w_{\mathcal{T}})\leq EG(w)$. If $\mathcal{T}$ is colorful, then $w_{\mathcal{T}} \in S_m$ for some $m \leq 4k-4$. \end{lem}

\begin{proof}
Up to relabeling rows and columns to account for flattening, the tree
$\mathcal{T}$ is a James-Peel tree for $D(w_{\mathcal{T}})$ (not
necessarily complete), so $k \leq
EG(w_{\mathcal{T}})$. Theorem~\ref{thm:patternthm} implies
$EG(w_{\mathcal{T}}) \leq EG(w)$.

Suppose $\mathcal{T}$ is colorful. The number of letters in
$w_{\mathcal{T}}$ is at most $|R(\mathcal{T})| + |C(\mathcal{T})|$.
Consider the vertex indexed by $D(v)$ in the full tree $JP(w)$ .  Say $v^i =
vt_{rs}t_{rj_i}$ are the transitions of $v$, with $j_1 > \cdots >
j_p$. The rows of $D(v)$ involved in row moves are $r, j_1, \ldots,
j_p$, and the columns involved in column moves are $v(s), v(j_1),
\cdots, v(j_p)$. However, $R_{r\to j_{1}}D(v) \simeq D(v)$ and
$C_{v(s)\to v(j_{p})}D(v) \simeq D(v)$ by Lemma~\ref{lem:equivalent}, so
these edges are contracted in the reduced tree, so row $j_{1}$ and
column $v(j_{p})$ will not contribute to $R(\mathcal{T})$ and
$C(\mathcal{T})$ respectively.  Thus, if a vertex $F$ (which is
equivalent to some $D(v)$) of $\mathcal{T}$ has $p$ children, the
edges leading down from $F$ contribute at most $p$ elements to each of
$R(\mathcal{T})$ and $C(\mathcal{T})$. Summing over all vertices,
\begin{align*}
|R(\mathcal{T})| + |C(\mathcal{T})| &\leq 2\deg (D(w)) + \sum_{\substack{F \in \mathcal{T}\\ F \neq D(w)}} 2(\deg(F) - 1)\\
&= 2\left[\sum_{F \in \mathcal{T}} \deg(F)\right] - 2|V(\mathcal{T})| + 2\\
&= 4|E(\mathcal{T})| - 2(|V(\mathcal{T})| - 1)\\
&= 2|E(\mathcal{T})|\\
&\leq 4k-4,
\end{align*}
with the last inequality by Lemma~\ref{lem:maxtreesize}. 
\end{proof}

In particular, taking $\mathcal{T} = RJP(w)$ in
Lemma~\ref{lem:subtreepattern} gives the following theorem.  

\begin{thm} \label{thm:skeleton} Any permutation $w$ contains a pattern $v \in S_m$ such that $EG(w) = EG(v)$, for some $m \leq 4\ EG(w) - 4$. \end{thm}

More generally, Lemma~\ref{lem:subtreepattern} lets us show that $k$-vexillarity is characterized by avoiding a \emph{finite} set of patterns.

\begin{thm} \label{thm:kvex} Say $w$ is a permutation with $EG(w) > k$. Then $w$ contains a pattern $v \in S_m$ such that $EG(v) > k$, for some $m \leq 4k$. \end{thm}

\begin{proof} By Lemma~\ref{lem:subtreepattern}, it suffices to
exhibit a colorful subtree of $RJP(w)$ rooted at $D(w)$ with $k+1$
leaves. Construct such a tree $\mathcal{T}$ as follows. First take
$\mathcal{T}$ to have only the vertex $D(w)$. Add the two children of
$D(w)$ corresponding to the $R$-edge and the $C$-edge.  Continue
adding the remaining children of $D(w)$ until $\mathcal{T}$ has $k+1$
leaves or all children have been added.  If all children of $D(w)$
have been added and $\mathcal{T}$ has fewer than $k+1$ leaves, then
since $RJP(w)$ has at least $k+1$ leaves, there is a leaf $F$ of
$\mathcal{T}$ with at least two children. Now repeat this process
starting with $F$ in place of $D(w)$. Iterating, eventually
$\mathcal{T}$ will have $k+1$ leaves, and is colorful by construction.
\end{proof}

\begin{cor} \label{cor:kvexbound} A permutation $w$ is $k$-vexillary if and only if it avoids all non-$k$-vexillary patterns in $S_m$ for $1 \leq m \leq 4k$. \end{cor}

For $k = 2$, we can explicitly find all non-$2$-vexillary patterns in $S_m$ for $1 \leq m \leq 8$ and eliminate those containing a smaller non-$2$-vexillary pattern to find a minimal list.

\begin{thm} \label{thm:2vexthm} A permutation $w$ is $2$-vexillary if and only if it avoids all of the following $35$ patterns.
\begin{equation*}
\begin{array}{ccccccc}
21543 & 231564   & 315264 & 5271436 & 26487153 & 54726183 & 64821537\\
32154 & 241365   & 426153 & 5276143 & 26581437 & 54762183 & 64872153\\
214365 & 241635  & 2547163 & 5472163 & 26587143 & 61832547 & 65821437\\
214635 & 312645  & 4265173 & 25476183 & 51736284 & 61837254 & 65827143\\
215364 & 314265  & 5173264 & 26481537 & 51763284 & 61873254 & 65872143
\end{array}
\end{equation*}
\end{thm}

This process is also feasible for $k = 3$, in which case we need to
look at non-3-vexillary patterns up through $S_{12}$. Here we find
that the bound in Corollary~\ref{cor:kvexbound} is not sharp.

\begin{thm} \label{thm:3vexthm} A permutation $w$ is $3$-vexillary if and only if it avoids a list of $91$ patterns in $S_6 \cup S_7 \cup S_8$.
For the full list of patterns, see
\end{thm}
\begin{center}
 \url{http://www.math.washington.edu/~billey/papers/k.vex.html}.
\end{center}

The 3-vexillary permutations have some interesting properties.  First,
in Section~\ref{sec:mult.free} we will show their Stanley symmetric
functions are always multiplicity free.  Second, their essential sets
are relatively simple.

In \cite{Fulton1}, Fulton defined the \textit{essential set} of a
permutation $w$, $Ess(w)$, to be the set of southeast corners of the
connected components of the diagram $D(w)$.  He showed that the rank
conditions for the Schubert variety indexed by $w$ need only be
checked at cells in the essential set.  Furthermore, he showed that
the essential set of a vexillary permutation has no two cells
$(i_{1},j_{1})$ and $(i_{2},j_{2})$ with $i_{1}< i_{2}$ and
$j_{1}<j_{2}$. Thus, the essential set lies along a lattice path going
from the southwest corner of the diagram to the northeast using only
north and east steps.  See \cite[Prop.4.6] {reinerwooyong} for an alternative
description of the essential set using minimal bigrassmannian elements
not below $w$ in Bruhat order.

One can characterize permutations whose essential set consists of two
nonintersecting such lattice paths in terms of pattern avoidance.

\begin{lem}\label{lem:essential}
A permutation $w$ has essential set with no three cells
$(i_{1},j_{1})$, $(i_{2},j_{2})$, and $(i_{3},j_{3})$ with $i_{1}<
i_{2}<i_{3}$ and $j_{1}<j_{2}<j_{3}$ if and only if $w$ avoids the 25 patterns
\begin{equation*}
\begin{array}{ccccc}
214365 & 3251746   & 35172864 & 35281746 & 53182764\\
2416375 & 3251764   & 35182746 & 35281764 & 53271846\\
2417365 & 4216375  & 35182764 & 53172846 & 53271864\\
3152746 & 4216735  & 35271846 & 53172864 & 53281746\\
3152764 & 35172846  & 35271864 & 53182746 & 53281764
\end{array}
\end{equation*}
\end{lem}

\begin{cor}\label{cor:essential.set}
If a permutation $w$ is 3-vexillary, its essential set does
not contain any three cells $(i_{1},j_{1}),(i_{2},j_{2}),
(i_{3},j_{3})$ with $i_{1}<i_{2}<i_{3}$ and $j_{1}<j_{2}<j_{3}$.  
\end{cor}

\begin{proof}
None of the patterns in Lemma~\ref{lem:essential} are $3$-vexillary, so this follows from Corollary~\ref{cor:kvexrespectspatterns}.
\end{proof}

\begin{rem} \label{rem:finite-LS-tree} The essential set can be used
to give a short proof that the Lascoux-Sch\"{u}tzenberger tree is
finite.  First, the L-S tree can contain only finitely many $w$ with
more than one maximal transition, e.g. because $F_w = \sum_{v} F_v$
for $v$ running over transitions of $w$, and the coefficient of $x_1
\cdots x_{\ell}$ in $F_w$ is always positive.  Second, if $Ess(w)$
lies in a single row, then $w$ is vexillary by Fulton's result above.
Otherwise, if $w$ is not vexillary but has exactly one maximal
transition $v = wt_{rs}t_{rj}$ where $r$ is the largest index of a
non-empty row in $Ess(w)$, then one can show by considering the
diagram of the permutation that $Ess(v) = Ess(w) \setminus \{(r,s)\}
\cup \{(r-1,s-1)\}$.  The same argument holds if $w$ must be replaced
by $1\times w$ in the algorithm.  Thus, either the distance between
the top and bottom (non-empty) rows of the essential set strictly
decreases upon passing from $w$ to $v$, or this distance remains the
same but the number of elements in the essential set in the bottom row
decreases.
\end{rem}

Searching through all non-4-vexillary permutations in $S_{16}$ is currently beyond our computational capabilities. However, one does find that every non-4-vexillary permutation in $S_{13}$ contains a proper non-4-vexillary pattern.

\begin{conj} A permutation $w$ is $4$-vexillary if and only if it avoids a list of $2346$ patterns in $S_6 \cup S_7 \cup \cdots \cup S_{12}$. \end{conj}

For the minimal list of non-4-vexillary patterns in $S_{13}$, see
\begin{center}
 \url{http://www.math.washington.edu/~billey/papers/k.vex.html}.
\end{center}

If one wants to compute or bound $EG(w)$, the Lascoux-Sch\"utzenberger tree is almost certainly more efficient than using our pattern characterizations. However, pattern characterizations lend themselves nicely to comparison, as exemplified in the proof of Corollary~\ref{cor:essential.set}. The connection to patterns also leads to enumerative results relating to $EG(w)$, since there has been much work done on enumerating permutations avoiding a given set of patterns, for example \cite{layered-perms-stanley-wilf-limit}.

\section{Diagram varieties}
\label{sec:diagvars}

Let $\Gr(k,n)$ denote the Grassmannian variety of $k$-planes in $\C^n$. For a diagram $D$ contained in a $k \times (n-k)$ rectangle, let $\Omega_D^\circ$ be the set of $k$-planes given as row spans of the matrices
\begin{equation*}
\{(I_k | A) : A \in M_{k \times (n-k)}, \text{$A_{ij} = 0$ if $(i,j) \in D$}\}.
\end{equation*}
Here $I_k$ is the $k \times k$ identity matrix. Let $\Omega_D$ be the closure of $\Omega_D^\circ$ in $\Gr(k,n)$. We call $\Omega_D$ the \emph{diagram variety} associated to $D$ (though it also depends on $k$ and $n$).

Recall that partitions contained in the rectangle $k \times (n-k)$ are in bijection with $k$-subsets of $[n]$. Specifically, $\lambda$ corresponds to the set
\begin{equation*}
B_{\lambda} = \{n-k+i-\lambda_i : 1 \leq i \leq k\}.
\end{equation*}
Write
\begin{equation*}
B_{\lambda} = \{b_1 < \cdots < b_k\} \text{\quad and\quad} [n] \setminus B_{\lambda} = \{c_1 < \cdots < c_{n-k}\},
\end{equation*}
and define a permutation $w_{\lambda}$ of $[n]$ in one-line notation by $w_{\lambda} = b_1 \cdots b_k c_1 \cdots c_{n-k}$.

Taking the standard basis $e_1, \ldots, e_n$ of $\C^n$, define a complete flag $F_{\bullet}$ by
\begin{equation*}
F_i = \langle e_{1}, \ldots, e_{i} \rangle.
\end{equation*}
The \textit{Schubert cell} is defined as 
\begin{equation*}
X_{\lambda}^\circ = \{X \in \Gr(k,n) : \dim(X \cap F_i) > \dim(X \cap F_{i-1}) \text{ if and only if $i \in B_{\lambda}$}\},
\end{equation*}
and its closure in the Zariski topology on $\Gr(k,n)$ is the
\textit{Schubert variety} $X_{\lambda}$ \cite{youngtableaux}.  The
codimension of $X_{\lambda}$ is $|\lambda |$ as defined.  In
particular, the diagram variety $\Omega_{\lambda}$ indexed by the
Ferrers diagram for $\lambda$ can be written as $\Omega_{\lambda}=
X_{\lambda} w_{\lambda}$ since right multiplication by a permutation
matrix permutes columns of the matrices in $X_{\lambda}$.  Thus
diagram varieties generalize the Schubert varieties up to change of
basis.

Let $\sigma_\lambda$ be the cohomology class in $H^{2|D|}(\Gr(k,n),
\Z)$ associated to $\Omega_\lambda $. One has the following classical
facts about the \emph{Schubert classes} $\sigma_{\lambda}$ (see
\cite{youngtableaux}).

\begin{itemize}
\item The classes $\sigma_{\lambda}$ for $\lambda$ varying over all partitions contained in $(k^{n-k})$ form a $\Z$-basis of $H^*(\Gr(k,n), \Z)$.

\item Let $\Lambda$ denote the ring of symmetric functions over $\Z$ in infinitely many variables. Then $\sigma_{\lambda} \mapsto s_{\lambda}$ defines an isomorphism of rings
\begin{equation*}
\phi : H^*(\Gr(k,n), \Z) \xrightarrow{\sim} \Lambda / \langle s_{\lambda} : \lambda \not\subseteq (k^{n-k})\rangle.
\end{equation*}
\end{itemize}

The second fact suggests a relationship to Specht modules. For example, consider the skew shape $\lambda \cdot \mu$ obtained by placing $\lambda$, $\mu$ together with no cell from $\lambda$ in the same row or column as a cell from $\mu$:
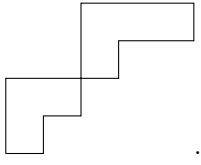
\begin{center}
\begin{tikzpicture}[scale=0.5]
\draw (0,0) -- (0,2) -- (2,2) -- (2,4) -- (5,4) -- (5,3) -- (3,3) -- (3,2) -- (2,2) -- (2,1) -- (1,1) -- (1,0) -- (0,0);
\end{tikzpicture}.
\end{center}
Suppose $\lambda \cdot \mu$ is contained in $(k^{n-k})$. The
multiplicity of the irreducible $S^{\nu}$ in $S^{\lambda \cdot \mu}$
is the Littlewood-Richardson coefficient $c_{\lambda\mu}^{\nu}$. This
is also the coefficient of $s_{\nu}$ in the Schur expansion of
$s_{\lambda\cdot \mu} = s_{\lambda}s_{\mu}$, hence the coefficient of
$\sigma_{\nu}$ in $\sigma_{\lambda \cdot \mu}$ \cite{Fulton1}.

Every closed subvariety of the Grassmannian has an associated
cohomology class \cite{Fulton}.  In particular, each diagram variety
$\Omega_{D}$ has an associated class $\sigma_{D}$ which can be
expressed as a symmetric function via $\phi$.  Liu studied diagram
varieties and their cohomology classes in \cite{liuthesis}, and made
the following conjecture, which generalizes the remarks above.

\begin{conj}[Liu's conjecture] \label{conj:diagramvariety}
Let $D \subseteq (k^{n-k})$.  If the generalized Schur function $s_{D}$ 
associated to $D$ expands into classical Schur functions as 
\begin{equation*}
s_D = \sum_{\lambda} c^D_{\lambda} s_{\lambda}, 
\end{equation*}
then the cohomology class for $\Omega_{D}$ has the same expansion coefficients 
\begin{equation*} 
\sigma_D = \sum_{\lambda} c^D_{\lambda} \sigma_{\lambda}.
\end{equation*}
Thus, the map $\phi$ sends $\sigma_D$ to $s_D$.
\end{conj}

\begin{rem}
One can show that if $D$ is contained in $(k^{n-k})$, so is any $\lambda$ with $S^{\lambda} \hookrightarrow S^D$ (this is obvious when $D$ has a complete James-Peel tree).
\end{rem}

Let $D^{\vee}$ denote the complement of $D$ in the rectangle $(k^{n-k})$. Conjecture~\ref{conj:diagramvariety} is known in several special cases.

\begin{thm}[{\cite[Proposition 5.5.3]{liuthesis}}] \label{thm:liuskew} Conjecture~\ref{conj:diagramvariety} holds when $D^{\vee}$ is a skew shape $\lambda/\mu$. \end{thm}

Given a diagram $D$ in $[k]\times [n-k]$, Liu constructs a bipartite
graph $G_{D}=([k],E,[n-k])$ where $E$ contains an edge $(i,j)$ if and
only if $(i,j)\in D$.   

\begin{thm}[{\cite[Theorem 5.4.3]{liuthesis}}] \label{thm:liuforests}
Conjecture~\ref{conj:diagramvariety} holds for a diagram $D$ provided
$G_{D^\vee}$ is a forest.
\end{thm}

A key tool in Liu's proof of Theorem \ref{thm:liuforests} is an analogue of Theorem~\ref{thm:JP}, albeit with a weaker conclusion. Given $\alpha_1, \alpha_2 \in H^*(\Gr(k,n))$, write $\alpha_1 \leq \alpha_2$ if $\alpha_2 - \alpha_1$ is a nonnegative linear combination of the Schubert classes $\sigma_{\lambda}$.

\begin{thm}[{\cite[Proposition 5.3.3]{liuthesis}}] \label{lem:liuJP}
Let $D$ be a diagram, and $(i_1, j_1), (i_2, j_2) \in D$ such that $(i_1, j_2), (i_2, j_1) \notin D$.
\begin{equation*}
\sigma_{(R_{i_1 \to i_2} D)^{\vee}},\, \sigma_{(C_{j_1 \to j_2} D)^{\vee}} \leq \sigma_{D^{\vee}}.
\end{equation*}
\end{thm}

Like the Schur function $s_D$, $\sigma_D$ only depends on $D$ up to equivalence.

\begin{lem} If $D, D'$ are equivalent diagrams, then $\sigma_D = \sigma_{D'}$. \end{lem}

\begin{proof} Permuting columns of $D$ corresponds to a change of
basis of $\C^n$, which does not change $\sigma_D$ since multiplication
by an element of $\GL_{n}$ induces a rational equivalence on varieties
in $\Gr(k,n)$ \cite{Fulton}.  As for rows, identify a permutation $v$
with a permutation matrix. If $(I|A)$ is a matrix representing a point
of $\Omega_D$, then $(I|vA)$ represents the same $k$-plane as
$(v^{-1}I|A)$, and so by permuting the first $k$ basis vectors
according to $v$, we see that $\sigma_D$ is not affected by permuting
rows of $D$.
\end{proof}

Liu proves a weaker result than Conjecture~\ref{conj:diagramvariety}
in the case of diagram varieties for the complement of a permutation
diagram.

\begin{prop}[{\cite[Proposition 5.5.4]{liuthesis}}] \label{thm:rightdegree} Under the Pl\"ucker embedding $\Omega_{D(w)^{\vee}} \hookrightarrow \Gr(k,n) \hookrightarrow \P^{{n \choose k} - 1}$, the degree of $\Omega_{D(w)^{\vee}}$ is $\dim S^{D(w)} = |\Red(w)|$.
\end{prop}

\begin{rem}
Note that this is what the degree must be if
Conjecture~\ref{conj:diagramvariety} is to hold. This is because
$\sigma_{(1)}$ is the class of a hyperplane intersected with
$\Gr(k,n)$ in the Pl\"ucker embedding, so the degree of
$\Omega_{D^{\vee}}$ is the coefficient of $\sigma_{(k^{n-k})}$ in
$\sigma_{D^{\vee}}\cdot \sigma_{(1)}^{|D|}$ \cite{Fulton}. When $D =
\lambda$ is a partition, it is easy to see using Pieri's rule that
this coefficient is the number of standard Young tableaux of shape
$\lambda$, which is the dimension of $S^{\lambda}$. The claim for
general $D$ follows by linearity.
\end{rem}

\begin{thm} If $w$ is multiplicity free, then Conjecture~\ref{conj:diagramvariety} holds for $\Omega_{D(w)^\vee}$. \end{thm}

\begin{proof}
Magyar \cite{magyarborelweil} showed that for any diagram $D$ (contained in a fixed rectangle), if $s_D = \sum_{\lambda} a_{\lambda} s_{\lambda}$ then $s_{D^\vee} = \sum_{\lambda} a_{\lambda} s_{\lambda^{\vee}}$. In particular, $s_{D(w)^\vee}$ is multiplicity free if $w$ is. Suppose $s_{\lambda^\vee}$ appears in $s_{D(w)^\vee}$. Then $\lambda$ is the image of $D(w)$ under a sequence of James-Peel moves, by Theorem~\ref{thm:permutationJPtree}. Theorem~\ref{lem:liuJP} then shows that $\sigma_{\lambda^\vee} \leq \sigma_{D(w)^\vee}$. Since $s_{D(w)^\vee}$ is multiplicity-free, this implies $\phi^{-1}(s_{D(w)^{\vee}}) \leq \sigma_{D(w)^\vee}$. Equality now follows from Proposition~\ref{thm:rightdegree}.
\end{proof}

Theorems~\ref{thm:liuskew} and \ref{thm:liuforests} prove Conjecture~\ref{conj:diagramvariety} when $D(w)$ is equivalent to a skew shape or a forest, and we note that these have nice statements in terms of pattern-avoidance as well. It is shown in \cite{billeyjockuschstanley} that if $w$ is $321$-avoiding, then $D(w)$ is equivalent to a skew shape. As for forests,

\begin{thm}
The graph $G_{D(w)}$ is a forest if and only if $w$ avoids $3412$, $4312$, $3421$, and $4321$.
\end{thm}

The permutations avoiding these four patterns have been studied by
Elizalde \cite{elizalde-almost-increasing-perms} in the context of
almost increasing permutations.

\begin{proof}
Clearly, if $D(w)$ has the property that the graph $G_{D(w)}$ is a forest, then so do all its
subdiagrams. Therefore, $w$ cannot contain $3412$, $4312$, $3421$, or $4321$, as one easily checks that none of these have graphs which
are forests.

For the converse, suppose that $G=G_{D(w)}$ is not a forest. Take a
sequence of distinct cells $b_1, \ldots, b_m \in D$ forming a cycle in
$G$.  Choose $i$ so that $b_i = (p,q)$ with $q$ maximal, then with $p$
maximal for that $q$. The three cells $b_{i-1}, b_i, b_{i+1}$ then form the pattern
$\begin{array}{cc}
        & \circ\\
\circ & \circ
\end{array}$
in $D(w)$. Since $D(w)$ is northwest, it therefore contains
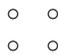
\begin{equation*}
\begin{array}{cc}
\circ & \circ\\
\circ & \circ
\end{array}
\end{equation*}
as a subdiagram.
After adding $\times$ to these rows and columns as usual for a permutation diagram, we must end up with one of the four following subdiagrams:
\begin{equation*}
\begin{array}{cccc}
\circ  & \circ & \times & \cdot \\
\circ  & \circ & \cdot  & \times\\
\times  & \cdot & \cdot  & \cdot \\
\cdot  &  \times & \cdot & \cdot
\end{array}\qquad
\begin{array}{cccc}
\circ  & \circ & \circ & \times \\
\circ  & \circ & \times  & \cdot\\
\times  & \cdot &  \cdot & \cdot \\
\cdot  &  \times & \cdot & \cdot
\end{array}\qquad
\begin{array}{cccc}
\circ  & \circ & \times & \cdot \\
\circ  & \circ & \cdot  & \times\\
\circ  & \times &  \cdot & \cdot\\
\times  &  \cdot & \cdot & \cdot
\end{array} \qquad
\begin{array}{cccc}
\circ  & \circ & \circ & \times \\
\circ  & \circ & \times  & \cdot\\
\circ  & \times & \cdot  & \cdot\\
\times  &  \cdot & \cdot & \cdot
\end{array}
\end{equation*}
Then in the positions of $w$ corresponding to these four rows one finds a
pattern $3412$, $4312$, $3421$, or $4321$.
\end{proof}

\section{Multiplicity Bounded Permutations}
\label{sec:mult.free}

We will say a permutation $w$ is \textit{multiplicity free} provided
all nonzero coefficients of the Stanley symmetric function $F_{w}$ are
1.  See A224287 in the OEIS for the number of multiplicity free
permutations in $S_{n}$ as a function of $n$.  By
Corollary~\ref{cor:kvexrespectsmultiplicity}, we know the multiplicity
free permutations respect pattern containment in the classical sense.
We discuss a new type of pattern containment which these permutations
also respect using the code of a permutation.  We follow up with
another variation on the theme of bounding the multiplicities in a
Stanley symmetric function which generalize vexillary permutations.

\begin{lem}\label{l:mult.free.3.vex}
Every 3-vexillary permutation is multiplicity free.
\end{lem}

\begin{proof}
Apply Lemma~\ref{lem:minmaxpartitions} to $D(w)$.
\end{proof}

The following conjecture has been tested through $S_{12}$ and one direction follows from Corollary~\ref{cor:kvexrespectsmultiplicity}. 
For the minimal list of 189 patterns up to $S_{11}$, see
\begin{center}
 \url{http://www.math.washington.edu/~billey/papers/k.vex.html}.
\end{center}

\begin{conj}\label{conj:mult.free}
The set of multiplicity free permutations is closed under taking
patterns and the minimal patterns all occur in $S_{n}$ for $n\leq 11$.
\end{conj}

Recall the inversion set of $w \in S_{n}$ is
\[
\mathrm{Inv}(w)=\{(i,j): 1\leq i< j\leq n \text{ and }  w_{i}>w_{j}\}.
\]
The \textit{code} of $w$ is the vector $\mathrm{code}(w)=(c_{1},\dots
, c_{n})$ such that $c_{k}$ is the number of inversions $(k,j)$ for
any $k<j\leq n$.  Equivalently, $c_{k}$ is the number of cells on row
$k$ of $D(w)$.  An inverse operation $\mathrm{code}^{-1}(c_{1},\dots
,c_{n})$ is obtained by multiplying out the reduced expression
$(s_{c_{1}} \dots s_{2}s_{1}) \cdot (s_{c_{2}+1} \dots s_{3}s_{2}) \cdots
(s_{c_{n}+n-1}\cdots s_{n}) $ where each factor is a consecutive
decreasing string of adjacent transpositions.  Adding additional zeros
at the beginning or the end of the code will not change the
corresponding Stanley symmetric function.  Furthermore, for every
vector of nonnegative integers there exists a permutation with this
vector as its code plus possibly some additional terminal 0's.

For example, $\mathrm{code}^{-1}(3,0,5,1) = (s_{3}s_{2}s_{1}) \cdot
(s_{7}s_{6}s_{5}s_{4}s_{3}) \cdot (s_{4}) = 4 1 8 3 2 5 6 7 $ and
$\mathrm{code}(4 1 8 3 2 5 6 7)=(3,0,5,1,0,0,0,0)$.

\begin{defn}\label{def:code.pattern}
We will say a permutation $w$ contains a permutation $v$ as a
\textit{simple code pattern} provided $\mathrm{code}(w)=(c_{1},\dots ,
c_{n})$, there exists an $i$ such that $c_{i}=0$, and
$\mathrm{code}^{-1}(c_{1},\dots , \widehat{c_{i}}, \dots , c_{n})=v$.
Say $w$ contains $v$ as a \textit{code pattern} provided there exists
a sequence of permutations $u^{(1)}, \ldots, u^{(k)}$ such that
$w=u^{(1)}$, $v=u^{(k)}$ and each $u^{(i)}$ contains $u^{(i+1)}$ as a
simple code pattern.    In this case, $|D(v)|=|D(w)|.$
\end{defn}

\begin{lem}\label{lem:code.pattern}
If $w$ contains $v$ as a code pattern, then $S^{D(v)} \hookrightarrow
S^{D(w)}$ as a submodule and $F_{w}-F_{v}$ is Schur positive.
\end{lem}

\begin{proof}
Without loss of generality, assume that $w \in S_n$ contains $v \in
S_{n}$ as a simple code pattern, and the code of $v$ is obtained from
the code of $w$ by removing $c_{k}=0$ and adding a zero at the end.
Since $c_{k}=0$, $w(k) = \mathrm{min}\{w(k),w(k+1),\dotsc , w(n)\}$.  Let $D$ be $D(w)$ with the empty row $k$ removed, so $S^{D} \simeq S^{D(w)}$.  


Recall that the code of a permutation is given by the number of
elements in the diagram on each row. Going from a code vector to the corresponding 
diagram is easy.  Starting at the first row, fill in the appropriate
number of cells left justified.  Place an $\times $ in the next
position and cross out everything below and to its right. For the next
row, starting from the leftmost available position that hasn't already
been crossed out, greedily place the appropriate number of cells
moving left to right.  Once the cells are placed in the row, put an
$\times$ in the next available position and cross out everything below
and to the right of the $\times$.  Continue until only fixed points are
added to the permutation.  Thus, the first $k-1$ rows and $w(k)-1$
columns of $D$ and $D(v)$ are identical since the codes of $v,w$ agree in the
first $k-1$ positions.

It remains to show that there exists a sequence of James-Peel moves taking $D$
to $D(v)$ which only modifies cells southeast of $(k,w(k))$.  Let
$j_{1}<j_{2}<\dotsb <j_{a}$ be the occupied columns of $D(w)$ southeast
of $(k,w(k))$.  Let $j_{0}=w(k)$.  Observe that $D(w)$ is empty in
column $j_{0}$ below row $k$ but may contain cells above row $k$.  We
claim that for $i>k$ and $1\leq l\leq a$, $(i,j_{l}) \in D(w)$ if and
only if $(i,j_{l-1}) \in D(v)$ by construction of the diagram from the
code.  So we can shift the occupied columns of $D(w)$ southeast of
$(k,w(k))$ over left by applying $C_{j_{0} \to j_{1}}$ to $D$ and
then applying $C_{j_{1} \to j_{2}}$, etc.  Furthermore, for $i<k$ if $(i,j_{l})
\in D$ then $(i,j_{l-1}) \in D$ and $D$ and $D(v)$ agree above row
$k$, so applying each $C_{j_{l}\to j_{l-1}} D$ does not change any
cells above row $k$.  Thus,
\[
D(v) =C_{j_{a} \to j_{a-1}} \cdots C_{j_{2} \to j_{1}} C_{j_{1} \to
j_{0}} D.
\]
We conclude that $S^{D(v)} \hookrightarrow S^{D(w)}$ by
Lemma~\ref{lem:JPtree}.   
\end{proof}

\begin{cor}\label{cor:code.pat}
Assume $w$ contains $v$ as a code pattern.  If $w$ is multiplicity
free, then so is $v$.
\end{cor}

Next we generalize multiplicity free permutations to a filtration of
permutations.

\begin{defn}\label{defn:k.coeff.bounded}
A permutation $w$ is $k$-\textit{multiplicity bounded} provided each
$a_{w\lambda} \leq k$ in the expansion $F_w = \sum_{\lambda}
a_{w\lambda} s_{\lambda}$.  Thus, 1-multiplicity bounded is the same as
multiplicity free.
\end{defn}

For each $k\geq 1$, the set of all $k$-multiplicity bounded
permutations respects pattern containment by
Corollary~\ref{cor:kvexrespectsmultiplicity}.  If one could bound the
size of the minimal patterns which are not $k$-multiplicity bounded,
then one would prove the following conjecture.

\begin{conj}
The $k$-multiplicity-bounded permutations are defined by avoiding a
finite set of permutation patterns.  
\end{conj}

\section{Future work}
\label{sec:future}


We were led to Theorem~\ref{thm:patternthm} by trying to study pattern
containment for diagrams.  In particular, we observed in experiments
that the conclusion of Corollary~\ref{cor:JPtreesubdiagram} holds for
arbitrary diagrams and subdiagrams.  Is this always true?
Corollary~\ref{cor:JPtreesubdiagram} holds when the subdiagram is
(equivalent to) a permutation diagram, a skew shape, or a
column-convex diagram, since these diagrams all admit complete
James-Peel trees. The algorithm given by Reiner and Shimozono in
\cite{percentavoiding} for decomposing Specht modules shows that the
conclusion of Corollary~\ref{cor:JPtreesubdiagram} also holds when $D$ is
percent-avoiding and $D' = D \cap \{i : a \leq i \leq b\} \times \{j :
c \leq j \leq d\}$ for some $a,b,c,d$.

We have no simpler characterizations of the lists of patterns arising
from Corollary~\ref{cor:kvexbound} and Theorems \ref{thm:2vexthm} and
\ref{thm:3vexthm}. One necessary condition for $w$ to be
non-$k$-vexillary but contain only $k$-vexillary patterns is that
every $w(i)$ participates in some $2143$ pattern. Otherwise, the $i$th
row and $w(i)$th column of $D(w)$ are contained in or contain every
other row and column, and so they do not participate in the James-Peel
moves of $RJP(w)$. This is far from sufficient, however. Magnusson and
\'Ulfarsson \cite{marked-mesh-learning} have developed an algorithm for
characterizing sets of permutations in terms of avoiding mesh
patterns, but this algorithm does not seem to simplify our patterns
appreciably. One might try even more general notions of patterns, such
as marked mesh patterns.  Bridget Tenner has noted that some
2-vexillary patterns do collapse.  In these cases though, the
algorithms for detecting pattern containment require checking for the
original patterns.

In \cite{typebvex}, vexillary elements of types $B, C, D$ in the
hyperoctahedral group are defined as those whose Stanley symmetric
function is equal to a single Schur $P$- or $Q$-function ($P$ in types
$B, D$, and $Q$ in type $C$), and it is shown that the vexillary
elements are again characterized by avoiding a finite set of
patterns. Computer calculations show that
Corollary~\ref{cor:kvexrespectspatterns} with $k = 2$ holds in $B_9$
for types $B, C$ and in $D_8$; moreover, the $2$-vexillary patterns in
$B_9$ of types $B, C$ are characterized by avoiding sets of patterns
in $B_3 \cup \cdots \cup B_8$. The main obstacle to extending our
proofs to these other root systems is the apparent lack of an analogue
of the Specht module of a diagram. In a recent preprint
\cite{andersonfultonvex}, Fulton and Anderson give a different
variation on vexillary permutations in types $B, C, D$, and one might
ask if there is a reasonable notion of $k$-vexillary in their setting.

Klein, Lewis and Morales have recently defined another generalization
of vexillary permutations.  For $w \in S_n$, let $D(w)$ be its
permutation diagram.  It is shown in \cite{KLM}, that the rows and columns of
$D(w)$ can be rearranged to form the complement of a  skew shape if and only if $w$
avoids 9 patterns.  They call these \textit{skew vexillary}
permutations.  Under what conditions can the rows and columns of an
arbitrary diagram be rearranged into a skew shape or the complement of
a skew shape?

\section*{Acknowledgments}

We would like to thank Dave Anderson, Andrew Berget, Alain Lascoux, Ricky Liu, Aaron
Pihlman, Austin Roberts, Mark Shimozono, Bridget Tenner, and Henning \'Ulfarsson for
helpful discussions, and Eric Peterson for the term ``bushy''.

\bigskip

 \bibliographystyle{plain}
 \bibliography{../k.vex/algcomb}

\end{document}